\documentclass[a4paper]{amsart}
\usepackage[utf8]{inputenc}
\usepackage{amssymb,amsmath,enumerate,amsthm,url}
\usepackage[all,cmtip]{xy}
\usepackage{graphicx}
\usepackage[numbers]{natbib}
\usepackage{bibentry}
\usepackage[bookmarks=true]{hyperref}

\usepackage{comment}
\usepackage[colorinlistoftodos]{todonotes}

\def\namedlabel#1#2{\begingroup
    #2%
    \def\@currentlabel{#2}%
    \phantomsection\label{#1}\endgroup
}

\usepackage[nameinlink,capitalize]{cleveref}
\crefformat{equation}{#2(#1)#3}

\theoremstyle{plain}
\newtheorem{theorem}{Theorem}[section]
\newtheorem{lemma}[theorem]{Lemma}
\newtheorem{proposition}[theorem]{Proposition}

\newtheorem{fact}[theorem]{Fact}
\newtheorem*{fact*}{Fact}
\crefname{fact}{Fact}{Facts}
\newtheorem{corollary}[theorem]{Corollary}
\newtheorem{claim}[theorem]{Claim}
\newtheorem*{claim*}{Claim}
\theoremstyle{definition}
\newtheorem{definition}[theorem]{Definition}
\newtheorem*{definition*}{Definition}

\newtheorem*{notation*}{Notation}
\theoremstyle{remark}
\newtheorem{remark}[theorem]{Remark}
\newtheorem*{remark*}{Remark}
\newtheorem{example}[theorem]{Example}
\newtheorem*{example*}{Example}

\newtheorem*{note*}{Note}
\newtheorem{question}[theorem]{Question}
\newtheorem*{question*}{Question}

\newcommand{\N}{\mathbb{N}}
\newcommand{\Q}{\mathbb{Q}}

\newcommand{\C}{\mathcal{C}}
\newcommand{\vc}{\operatorname{vc}}
\newcommand{\tp}{\operatorname{tp}}
\newcommand{\fin}{\operatorname{fin}}
\newcommand{\cl}{\operatorname{cl}}
\newcommand{\U}{\mathcal{U}}
\newcommand{\M}{\mathcal{M}}
\newcommand{\opp}{\operatorname{opp}}
\newcommand{\Maj}{\operatorname{Maj}}
\newcommand{\acl}{\operatorname{acl}}
\newcommand{\dcl}{\operatorname{dcl}}
\newcommand{\ind}{\operatorname{ind}}
\newcommand{\eq}{{\operatorname{eq}}}
\newcommand{\fs}{\operatorname{fs}}
\newcommand{\Sfs}[2]{S_{\text{$#1$-fs}}(#2)}

\newcommand{\Sdfs}[3]{S^{#1}_{\text{$#2$-fs}}(#3)}
\newcommand{\Sdfsl}[4]{S^{#1}_{{#4},\text{$#2$-fs}}(#3)}

\providecommand{\Th}{\operatorname{Th}}
\providecommand{\set}[2]{\{#1 \mid #2\}}
\providecommand{\sequence}[2]{(#1)_{#2}}
\providecommand{\varsequence}[2]{(#1\mid #2)}

\renewcommand{\SS}{\mathcal{P}}

\providecommand{\phiopp}{{\phi^{\opp}}}

\newcommand{\defn}[1]{{\bf #1}}
\newcommand{\recall}[1]{{\it #1}}

\newcommand{\calF}{\mathcal F}

\renewcommand{\L}{\mathcal L}
\renewcommand{\S}{\mathcal S}
\providecommand{\eps}{\epsilon}
\providecommand{\tuple}[1]{\overline{#1}}
\providecommand{\bb}{\tuple{b}}
\providecommand{\xx}{\tuple{x}}
\providecommand{\yy}{\tuple{y}}
\providecommand{\zz}{\tuple{z}}
\providecommand{\ee}{\tuple{\epsilon}}

\providecommand{\acleq}{\acl^{\eq}}
\providecommand{\dcleq}{\dcl^{\eq}}

\providecommand{\pfin}{\mathcal P^{\fin}}
\providecommand{\Ffin}{\mathcal F^{\fin}}

\providecommand{\kc}{k_{\operatorname{comp}}}

\newcommand{\Bvc}{B_{\vc}}
\newcommand{\Aut}{\operatorname{Aut}}

\newcommand{\ravg}[1]{\operatorname{rAvg}#1}
\newcommand{\ravga}[2]{\operatorname{rAvg}_{#2}#1}

\newcommand{\Sh}{\text{Sh}}

\newcommand{\Npq}{N_{\operatorname{pq}}}
\newcommand{\ksd}{k_{\operatorname{sd}}}

\begin{document}

\title{Density of compressible types and some consequences}
\author{Martin Bays}
\address{Institut für Mathematische Logik und Grundlagenforschung, Fachbereich 
Mathematik und Informatik, Universität Münster, Einsteinstrasse 62, 48149 
Münster, Germany}
\email{mbays@sdf.org}
\author{Itay Kaplan}
\address{Einstein Institute of Mathematics, Hebrew University of
Jerusalem, 91904, Jerusalem Israel.}
\email{kaplan@math.huji.ac.il}

\author{Pierre Simon}
\address{Dept.\ of Mathematics, University of California, Berkeley, 970 Evans 
Hall \#3840, Berkeley, CA 94720-3840 USA}
\email{simon@math.berkeley.edu}
\thanks{Bays was partially supported by DFG EXC 2044–390685587 and ANR-DFG 
AAPG2019 (Geomod).
Kaplan would like to thank the Israel Science Foundation for their
support of this research (grant no. 1254/18).
Simon was partially supported by the NSF (grants no. 1665491 and 1848562).}

\begin{abstract}
  We study compressible types in the context of (local and global) NIP. By extending a result in machine learning theory (the existence of a bound on the recursive teaching dimension), we prove density of compressible types. Using this, we obtain explicit uniform honest definitions for NIP formulas (answering a question of Eshel and the second author), and build compressible models in countable NIP theories.
 \end{abstract}

\maketitle

\section{Introduction}

By the Sauer-Shelah lemma, if a formula $\phi(x;y)$ is NIP, then the number of $\phi$-types over a finite set $A$ is bounded by a polynomial in the cardinality of $A$. For a stable formula, this is a consequence of definability of types: one only needs to specify the parameters involved in the definition. In dense linear orders, the reason for this phenomenon is different: for any finite set $A$ and element $b$, the $\leq$-type of $b$ over $A$ is implied by its restriction to some subset $A_0$ of size 2: the information of the full type can be \emph{compressed} down to this subset of bounded size. A $\leq$-type over an infinite set cannot in general be compressed down to a finite set, however finite parts of it can be uniformly compressed; following \cite{simon-decomposition}, we call such a type \emph{compressible} (\cref{def:compressible types prelim}). We expect NIP formulas to exhibit a combination of those two behaviours.

For NIP theories one manifestation of this philosophy is the result from \cite{simon-decomposition} that an arbitrary type has a generically stable part up to which it is compressible. Distal structures are (NIP) structures in which every type is compressible and hence this decomposition is trivial. For stable theories, compressible types turn out (\cref{lem:comp stable type}) to be precisely types which are \emph{l-isolated}, that is, isolated formula by formula. These play a role in Shelah's classification theory; one key property is that in a countable stable theory, an l-atomic model exists over any set \cite[IV.2.18(4),3.1(5),3.2(1)]{Sh-CT}. In this paper, we think of compressibility as an isolation notion and investigate its properties by analogy with the stable case. In order to obtain similar model-construction results, we need two basic properties: density of compressible types and transitivity of compressibility.

Density of compressible types over a set $A$ means that every formula over $A$ extends to a complete compressible type over $A$. We prove this for countable NIP theories (\cref{cor:density of comp in countable NIP}) by first considering the local setting of a single NIP formula $\phi$, and showing that any finite partial $\phi$-type extends to a complete compressible $\phi$-type (\cref{cor:localComp}). This is a combinatorial argument based on the proof by Chen, Cheng, and Tang \cite{CCT} of a bound on the ``recursive teaching dimension'' of a finite set system in terms of its VC-dimension. The existence of such a bound was used in \cite{UDTFS} to prove uniform definability of types over finite sets (UDTFS) for an NIP formula in an arbitrary theory. We generalise this result (answering \cite[Question~28]{UDTFS}) by showing uniformity of honest definitions for NIP formulas, which was previously known only assuming NIP for the whole theory \cite[Theorem~11]{CS-extDef2}. For this, we first show that an arbitrary $\phi$-type $p$ is a \emph{rounded average} of finitely many compressible types (\cref{thm:rounded average of compressible}). The rounded average of the compression schemes of these types gives an honest definition for $p$.

In fact, it turns out that full consistency of $p$ is not required here: for large enough $k$ we get uniform honest definitions for $k$-consistent families of instances of $\phi$, which we dub \emph{$k$-hypes} (\cref{cor:honest definitions for k-hypes}). Using this, we also obtain in \cref{thm:pseudofinite types are definable} uniform definability of $\phi$-types which are \emph{pseudofinite} in the sense that their positive and negative parts are pseudofinite (\cref{def:pseudofinite type}).

In order to prove transitivity, namely that $\tp(AB/C)$ is compressible when $\tp(A/BC)$ and $\tp(B/C)$ are, we return to the global setting of an NIP theory and use the type-decomposition theorem from \cite{simon-decomposition}. We show in \cref{prop:comp over smaller set with constants} that compressibility can be \emph{rescoped} to an arbitrary subset of the domain: if $\tp(a/B)$ is compressible and $C\subseteq B$, then $\tp(a/C)$ is compressible in the language with constants for elements of $B$. We deduce transitivity in \cref{prop:transitivity}.

Finally, we conclude that for countable NIP theories (or even countable theories naming any set of constants) one can construct models which are compressible over arbitrary sets (\cref{prop:comp constructible,prop:comp constructible atomic}). We give several applications:

$\bullet$ Given a definable unary set $X$ whose induced structure is stable, and any model $N$ of the theory of the induced structure, there is a model $M$ of $T$ such that $X(M)=N$ and moreover, if $N' \succ N$ then there is $M' \succ M$ such that $X(M')=N'$. This is \cref{cor:reduct map surjective}.

$\bullet$ If the theory is not stable, we can extend models without realising any non-algebraic generically stable type (\cref{cor:extend Preserve Stable,rem:extend preserve stable omega-sat}).

$\bullet$ We analyse compressiblity in ACVF, showing that then any model $M$ containing $A$ whose residue field is algebraic over $A$ is compressible over $A$ (\cref{exa:ACVF}).


\subsection{Acknowledgements}
We thank to Nati Linial and Shay Moran for answering a question that turned out 
to be precisely about the existence of a bound for the recursive teaching 
dimension, introducing us to this notion and to \cite{k-isolationQuadratic}.

We also thank Timo Krisam for helpful conversation which led to an improvement in the formulation of \cref{sec:decompComp}, and Eran Alouf for asking questions that led to \cref{thm:pseudofinite types are definable}.

Additionally we thank Anand Pillay and Martin Hils for suggesting that we consider the problem that led us to \cref{cor:reduct map surjective}.

Furthermore, we thank the anonymous referee for their careful reading of the manuscript and their many useful and precise comments which improved the presentation of the paper.
\section{Preliminaries}

\subsection{Languages, formulas and types} \label{sec:notations}
Our notation is standard. We use $\L$ to denote a first order language and $\phi(x,y)$ to denote a formula $\phi$ with a partition of (perhaps a superset of) its free variables. When $x$ is a (possibly infinite) tuple of variables and $A$ is a set contained in some structure (perhaps in a collection of sorts), we write $A^{x}$ to denote the tuples of the sort of $x$ (and of length $|x|$) of elements from $A$; alternatively, one may think of $A^{x}$ as the set of assignments of the variables $x$ to $A$. When $M$ is a structure and $A \subseteq M^x$, $b\in M^y$, we define $\phi(A,b)=\set{a\in A}{M \vDash \phi(a,b)}$.

$T$ will denote a complete theory in $\L$ (we do not really need $T$ to be complete, but it is more convenient), and $\U \vDash T$ will be a monster model (a sufficiently large saturated model\footnote{There are set theoretic issues in assuming that such a model exists, but these are overcome by standard techniques from set theory that ensure the generalised continuum hypothesis from some point on while fixing a fragment of the universe. The reader can just accept this or alternatively assume that $\U$ is merely $\kappa$-saturated and $\kappa$-strongly homogeneous for large enough $\kappa$.}). The word \emph{small} means ``of cardinality $<|\U|$''. As usual, we will assume that all models, tuples and sets of parameters are small and are contained in (perhaps a collection of sorts from) $\U$ unless stated otherwise. Some results, such as \cref{thm:rounded average of compressible}, hold for any set, by considering a bigger monster model and applying the result there.

When $B \subseteq \U$, $\L(B)$ is the language $\L$ augmented with constants for elements from $B$, and $\U_B$ is the natural expansion of $\U$ to $\L(B)$. A \recall{partial type} in variables $x$ (perhaps infinite, perhaps from different sorts) over $B \subseteq \U$ is a set of $\L(B)$-formulas in $x$ consistent with $\Th(\U_B)$ (i.e., formulas over $B$). For a  partial type $\pi$ over $B$ and $C \subseteq B$, we use the notation $\pi |_C$ for the \recall{restriction} of $\pi$ to $C$, namely all formulas $\phi(x) \in \L(C)$ implied by $\pi$ (i.e., $\pi \vdash \phi(x)$). Similarly, if $x'$ is a sub-tuple of $x$, the restriction of $\pi$ to $x'$ is the partial type consisting of all formulas in $x'$ implied by $\pi$.

A (complete) type over $B$ is a maximal partial type  over $B$. We denote the space of types over $B$ in $x$ by $S^x(B)$. It is a compact Hausdorff topological space in the logic topology (a basic open set has the form $\set{p \in S^x(B)}{\phi(x) \in p}$). For $a \in \U^x$, write $\tp(a/B)$ for the type of $a$ over $B$. $S(B)$ is the union of all types over $B$.

For an $\L$-formula $\phi(x,y)$, an \recall{instance} of $\phi$ over $B \subseteq  \U$ is a formula $\phi(x,b)$ where $b \in B^y$, and a (complete) \recall{$\phi$-type} over $B$ is a maximal partial type consisting of instances and negations of instances of $\phi$ over $B$. We write $S_{\phi}(B)$ 
for the space of $\phi$-types over $B$ in $x$ (in this notation we keep in mind the partition $(x,y)$, and $x$ is the first tuple there). As above, it is a compact Hausdorff topological space in the logic topology. For $a \in \U^x$, we write $\tp_\phi(a/B) \in S_\phi(B)$ for its $\phi$-type over $B$. We also use the notation $\phi^1=\phi$ and $\phi^0 = \neg \phi$. When $p(x)\in S(B)$ is a type, we write $p\restriction \phi \in S_\phi(B)$ for the complete $\phi$-type over $B$ implied by $p$. If $\Delta$ is a set of partitioned formulas $\phi(x,y)$, we define $S_\Delta(B)$ and the restriction $p\restriction \Delta \in S_\Delta(B)$ similarly.

We will also consider the case where $B \subseteq \U^y$ and (abusing notation) define $S_\phi(B)$ similarly --- this should never cause a confusion.

If $\pi(x)$ is a small partial type (over some small set contained in $\U$), we write $S_\phi^\pi(B)$ for the closed subspace of $S^x_\phi(B)$ consisting of the $\phi$-types which are consistent with $\pi$.

Generally we do not limit our discussion to finite tuples of variables (but in the context of $\phi$-types for a formula $\phi$ this does not matter).

We write $A \subseteq_{\fin}  B$ to mean that $A$ is a finite subset of $B$.

\subsection{Global and invariant types}
For $A \subseteq \U$, an \defn{$A$-invariant} type is a \recall{global} type, i.e., a type over $\U$, which is invariant under the action of $\Aut(\U/A)$, the group of automorphisms fixing $A$ pointwise.

For a sequence $\sequence{x_i}{i \in I}$ and $j \in I$, we write $x_{< j}$ for $\sequence{x_i}{i<j}$, and similarly for $x_{\leq j}$.

\begin{definition}
  If $q(x)$ and $r(y)$ are $A$-invariant global types, then the type $(q \otimes r)(x,y)$ is defined to be $\tp(a,b/\U)$ (in a bigger monster model) for any $b \vDash r$ and $a \vDash q|_{\U b}$ (here we understand $q$ to mean its unique $A$-invariant extension to a bigger model). (This can also be defined without stepping outside of the monster model, see \cite[Chapter~2]{simon-NIP}.)

  We define $q^{(n)}(x_{<n})$ for $n <\omega$ by   induction:  $q^{(1)}(x_0) = q(x_0)$,
  \[q^{(n+1)}(x_{\leq n})= q(x_{n}) \otimes q^{(n)}(x_{<n}),\]
  and $q^{(\omega)}(x_{<\omega})=\bigcup_{n<\omega}q^{(n)}$.

  For any linear order $(X,<)$, we can define $q^{(X)}\varsequence{x_i}{i\in X}$ similarly, as the union of $q^{(X_0)}\varsequence{x_i}{i\in X_0}$ for every finite $X_0 \subseteq X$.

\end{definition}

\begin{fact}\cite[Chapter~2]{simon-NIP} \label{fac:sequences generated by invariant types}
  Given a global $A$-invariant type $q$ and a linear order $(X,<)$, $q^{(X)}$ is an $A$-invariant global type.  In addition, it is the type of an indiscernible sequence over $\U$.

  For any small set $B \supseteq A$, $q^{(X)}|_B$ is given by $\tp(\varsequence{a_i}{i\in X} / B)$ where $a_i \vDash q|_{Ba_{<i}}$. This is a \defn{Morley sequence of $q$ over $B$} (indexed by $X$).
\end{fact}

\subsection{VC-dimension and NIP}

\begin{definition} [VC-dimension]
  Let $X$ be a set and $\mathcal{F}\subseteq\SS(X)$. The pair $(X,\mathcal{F})$ is called a \defn{set system}.
  We say that $A\subseteq X$ is \defn{shattered} by $\mathcal{F}$ if for every $S\subseteq A$ there is $F\in\mathcal{F}$ such that $F\cap A=S$. A family $\mathcal{F}$ is said to be a \defn{VC-class} on $X$ if there is some $n<\omega$ such that no subset of $X$ of size $n$ is shattered by $\mathcal{F}$. In this case the \defn{VC-dimension of $\mathcal{F}$}, denoted by $\vc(\mathcal{F})$, is the smallest integer $n$ such that no subset of $X$ of size $n+1$ is shattered by $\mathcal{F}$.

  If no such $n$ exists, we write $\vc(\mathcal{F})=\infty$.
\end{definition}

\begin{definition} \label{def:NIP formula, theory}
  Suppose $T$ is an $\L$-theory and $\phi(x,y)$ is a formula.
  Say $\phi(x,y)$ is \defn{NIP} if for some/every $M\vDash T$, the family $\set{\phi(M^x,a)}{a\in M^{y}}$ is a VC-class.

  The theory $T$ is \defn{NIP} if all formulas are NIP.
\end{definition}

\begin{definition}
  Suppose $T$ is an $\L$-theory and $\phi(x,y)$ is an NIP formula. Let $\vc(\phi)$ be the VC-dimension of $\set{\phi(M^x,a)}{a\in M^{y}}$, where $M$ is any (some) model of $T$.  Note that this definition depends on the partition of variables.

  Let $\phiopp$ be the partitioned formula $\phi(y,x)$ (it is the same formula with the partition reversed). Let $\vc^*(\phi) = \vc(\phiopp)$ be the \defn{dual} VC-dimension of $\phi$.
\end{definition}

\begin{fact} \cite[Lemma~6.3]{simon-NIP} \label{fact:dual}
  Suppose $\mathcal{F}$ is a VC-class on $X$. Let $\mathcal{F}^* = \set{\set{s \in \mathcal F}{x \in s}}{x \in X} \subseteq \SS(\mathcal{F})$ be the \defn{dual} of $\mathcal{F}$. Then $\mathcal{F}$ is a VC-class iff $\mathcal{F}^*$ is, and moreover $\vc^*(\mathcal{F}):=\vc(\mathcal{F}^*)<2^{\vc(\mathcal{F})+1}$.
\end{fact}

\begin{remark}\label{rem:bound on vc-dimension of dual}
  By \cref{fact:dual}, $\phi$ is NIP iff $\phiopp$ is NIP, and $\vc(\phiopp)=\vc^*(\phi)<2^{\vc(\phi)+1}$.
\end{remark}

By \cite[Lemma~2.9]{simon-NIP}, a Boolean combination of NIP formulas is NIP. In particular, if $\phi_i(x,y_i)$ is NIP for $i<k$ then so is $\bigwedge_{i<k} \phi_i(x,y_i)$. We end this subsection by giving an explicit bound on its VC-dimension; see also \cite[Theorem~9.2.6]{MR876078}. (This will be used only in \cref{sec:rounded average proof}.).

\begin{definition}\label{def:binom vc}
  Let $\binom s{\leq k} = \sum_{i\leq k} \binom s i$, and let $\Bvc(n,k) = \max \set{s \in \N}{\binom s{\leq k}^n \geq  2^s}$.
\end{definition}

\begin{remark} \label{rem:Bvc(n k) >= n k}
  For all $1 \leq n,k\in \N$, $\Bvc(n,k) \geq n,k$.

  Indeed, $\binom n {\leq k} \geq 2$, so $\binom n {\leq k}^n \geq 2^n$. Hence $\Bvc(n,k) \geq n$.
  Similarly, to show that $\Bvc(n,k) \geq k$, note that $\binom k {\leq k}^n = 2^{kn} \geq 2^k$.
\end{remark}

\begin{lemma} \label{lem:vcBool}
  Let $k \in \N$, and let $\phi_1(x;y_0),\ldots ,\phi_n(x;y_{n-1})$ be partitioned formulas with $\vc(\phi_i) \leq  k$.
  Let $\theta(x;y_0,\ldots ,y_{n-1}) = \bigwedge_{i<n} \phi_i(x,y_i)$.
  Then $\vc(\theta) \leq  \Bvc(n,k)$.
\end{lemma}
\begin{proof}
  Let $s > \Bvc(n,k)$,
  and let $A_0 \subseteq  \U^x$ with $|A_0| = s$.
  For each $i$, by Sauer-Shelah \cite[Lemma~6.4]{simon-NIP},
  at most $\binom s{\leq k}$ subsets of $A_0$ are defined by instances of $\phi_i$; hence at most $\binom s{\leq k}^n$ are defined by instances of $\theta$. It follows from the definition of $\Bvc(n,k)$ that $\theta$ does not shatter $A_0$.
\end{proof}

\subsection{Compressible types} \label{sec:compressible types}
Here we will review the basic properties of compressible types.

\begin{definition}\label{def:compressible types prelim}
  A type $p(x) \in S(A)$ is \defn{compressible} if for any formula $\phi(x,y)$ there is a formula $\psi(x,z)$ such that for every finite set $A_0 \subseteq A$, there is some $c \in A^z$ such that
  \begin{itemize}
    \item $\psi(x,c) \in p$ and
    \item $\psi(x,c) \vdash (p \restriction \phi)|_{A_0}$ (i.e., it implies the set $\set{\phi(x,a)}{\phi(x,a)\in p,\; a \in A_0^y} \cup \set{\neg \phi(x,a)}{\phi(x,a)\notin p,\; a \in A_0^y}$).
  \end{itemize}
\end{definition}

Suppose $A \subseteq \U$. Given $a \in \U^x$ any tuple, we let $(A,a)$ be the structure with universe $A$ and the induced structure coming from $a$-definable sets. In other words, for every formula $\phi(x,y)$, there is a relation $R_\phi(y)$ interpreted by $R_\phi(c)$ iff $\U \vDash \phi(a,c)$ for any $c\in A^y$. Note that if $M \equiv (A,a)$ then $M \cong (A',a)$ for some $A' \subseteq \U$, and moreover if $M \succ (A,a)$ then there is such an $A' \subseteq \U$ such that $M$ and $(A',a)$ are isomorphic over $A$. 
Thus, whenever we have such a structure, we will always assume it has the form $(A',a)$ for some $A' \subseteq \U$.

This construction preserves useful information on the type $\tp(a/A)$. For example, recall that a type $p(x) \in S(A)$ is \recall{definable} if for every formula $\phi(x,y)$, the set $\set{a \in A^y}{\phi(x,a)\in p}$ is definable over $A$. It is easy to see that if $\tp(a/A)$ is definable and $(A',a') \equiv (A,a)$ then $\tp(a'/A')$ is also definable (with the same definition scheme). Moreover, we have:

\begin{fact}\label{fac:compressiblity preserved under elem eq}
  \cite[Lemma~3.2]{simon-decomposition} If $\tp(a/A)$ is compressible and $(A',a') \equiv (A,a)$, then so is $\tp(a'/A')$.
\end{fact}

Compactness gives the following equivalent definition of compressibility:
\begin{fact}\label{fac:compressible in terms of elementary extension}
  The type $p=\tp(a/A)$ is compressible iff for any (some) $|A|^+$-saturated elementary extension $(A',a) \succ (A,a)$ and any formula $\phi(x,y)$, there is some formula $\psi(x,z)$ and $d \in (A')^z$ such that $\psi(a,d)$ holds and $\psi(x,d) \vdash (p \restriction \phi)|_A$.
\end{fact}

In fact, this was the original definition of compressibility in \cite[Definition~3.1]{simon-decomposition}.

We give another useful characterisation of compressible types.
Recall that two types $p(x)$, $q(y)$ over $A$ are \recall{weakly orthogonal} if $p \cup q$ implies a complete type in $x,y$ over $A$.

A type $q$ is \recall{finitely satisfiable} in some set $A$ if every formula from $q$ is realised in $A$. We write $\Sdfs{x}{A}{B} \subseteq  S^x(B)$ for the subspace consisting of those types in $x$ which are finitely satisfiable in $A$. Recall that such types can be extended to global types in $\Sdfs{x}{A}{\U}$ (using ultrafilters). Note that $\Sdfs{x}{A}{B}$ is a closed (and hence compact) subspace of $S^x(B)$. As usual, omitting the $x$ means taking all types (allowing infinite (small) tuples).

\begin{fact}\cite[Lemma~3.3]{simon-decomposition} \label{fac:compressible iff weakly orthogonal}
  ($T$ arbitrary) The following are equivalent for $a,A$ and an $|A|^+$-saturated extension $(A',a) \succ (A,a)$:
  \begin{enumerate}
    \item $\tp(a/A)$ is compressible.
    \item For all $q(y) \in \Sfs{A}{A'}$, $\tp(a/A')$ (as a type in $x$) and $q(y)$ are weakly orthogonal.
    \item For all $q(y) \in \Sfs{A}{A'}$, $\tp(a/A')$ (as a type in $x$) and $q(y)$ imply a complete type in $xy$ over $\emptyset$.
  \end{enumerate}
\end{fact}

\begin{corollary} \label{cor:orthFSComp}
  ($T$ arbitrary)
  A type $p(x) = \tp(a/A)$ is compressible if and only if
  there is some (possibly infinite) $d \subseteq  \U$ of length $|d| \leq  |T|$ such that
  \begin{itemize}\item $\tp(d/Aa)$ is finitely satisfiable in $A$, and
  \item for every $q \in \Sfs{A}{\U}$,
  \[q|_{Ad} \vdash  q|_{Aa}.\]
  \end{itemize}
\end{corollary}

\begin{proof}
  Suppose $\tp(a/A)$ is compressible.
  Let $\phi(x,y,w)$ be an $\L$-formula.
  Then there is $\zeta(x,z)$ such that $\set{\bigvee_{\eps<2} \forall x(\zeta(x,z) \rightarrow  \phi(x,c',c)^\eps)}{c' \in A^y, c \in A^w } \cup \{ \zeta(a,z) \}$ is finitely satisfiable in $A$, so let $e_\phi$ realise a completion in $\Sdfs{z}{A}{A}$.
  If $q \in \Sdfs{y}{A}{\U}$ and $c \in A^w$, then $q|_{Ae_\phi}(y) \vdash  \bigvee_{\eps<2} \forall x (\zeta(x,e_\phi) \rightarrow  \phi(x,y,c)^\eps)$ by finite satisfiability, and $\vDash  \zeta(a,e_\phi)$, so
  $q|_{Ae_\phi}(y) \vdash  \phi(a,y,c)^\eps$ for some $\eps<2$.
  So $d := \sequence{e_\phi}{\phi(x,y,w) \in \L}$ is as required.

  The other direction follows from \cref{fac:compressible iff weakly orthogonal}(3$\Rightarrow$1), since by saturation we can assume $d \subseteq A'$.
\end{proof}

\section{Density of (local) compressibility}

Here we will prove that (local) compressible types are dense. In \cref{sec:infinite CCT} we prove an abstract version of this dealing with set systems of finite VC-dimension (generalising \cite[Lemma~4]{CCT} to infinite sets).
Then in \cref{sec:local compression} we deduce that locally compressible types are dense for NIP formulas, and in \cref{sec:density of compressible in ctbl NIP} we deduce that compressible types are dense in countable NIP theories.

\subsection{Compressibility for set systems of finite VC-dimension}\label{sec:infinite CCT}

Let $A$ be a (possibly infinite) set. As usual, $2^A$ is the (Hausdorff compact) space of functions $A\to 2 = \{0,1\}$ equipped with the product topology. Any $\C \subseteq 2^A$ naturally induces a set system on $A$ (those sets whose characteristic functions are in $\C$) and as such has a VC-dimension $\vc(\C)$. For $\C \subseteq 2^A$ and $B\subseteq A$, let $\C|_B := \set{c|_B}{c \in \C}$, the set of restrictions to $B$.

Let $\C \subseteq  2^A$.

\begin{definition} \label{def:compressible Boolean functions}
  \begin{itemize}\item For $B \subseteq  A$ and $c' \in \C|_B$, define the relativisation $\C_{c'} := \set{c \in \C}{c|_{B} = c'} = \set{c \in \C}{c \supseteq c'}$.

  \item For $c \in \C$ and $B,C \subseteq  A$, write
  $c|_B \vdash _\C c|_C$ to mean that $c'|_C = c|_C$ for any $c' \in \C$ with $c'|_B = c|_B$.

  \item For $k<\omega$, say $c \in \C$ is \defn{$k$-compressible in $\C$}
  if for any finite $A_0 \subseteq  A$ there exists $A_1 \subseteq  A$ with $|A_1| \leq  k$ such that $c|_{A_1} \vdash _\C c|_{A_0}$.

  \item Say $c$ is \defn{compressible} in $\C$ if it is $k$-compressible for some $k<\omega$.
  \end{itemize}
\end{definition}

\begin{remark}
  This terminology is originally inspired by, but does not precisely agree with, the terminology around compression schemes in the statistical learning literature.
\end{remark}

\begin{remark}\label{rem:cofinal finite colouring}
  Suppose $(P,\leq)$ is a directed partial order, and $c:P \to r<\omega$ is some colouring. Then there is some  subset $X\subseteq P$ which is monochromatic ($X \subseteq c^{-1}(i)$ for some $i<r$) and cofinal (for all $p\in P$ there is some $q \in X$ such that $q \geq p$).

  Indeed, if not, then for every $i<r$ there is some $p_i \in P$ such that $c(q) \neq i$ for all $q \geq p_i$. Let $p \geq p_i$ for all $i<r$. Then $c(p) \neq i$ for all $i<r$, contradiction.
\end{remark}

The proof of the following theorem is an adaptation to the case of infinite $A$ of the proof of \cite[Lemma~4]{CCT}, which proves it for finite $A$ with the same bound on $k$.

\begin{theorem} \label{thm:infinite CCT compression}
  For any $d<\omega$, let $\kc(d) := 2^{d+1}(d-2)+d+4$.
  For any set $A$, if $\C \subseteq  2^A$ is closed, non-empty, and has VC-dimension $\leq d$,
  then there exists $c \in \C$ which is $\kc(d)$-compressible in $\C$.
\end{theorem}
\begin{proof}
  The proof is by induction on $d$.

  If $\vc(\C)=0$, then $\C$ is a singleton $\{ c \}$, and $c$ is clearly $0$-compressible in $\C$.

  Suppose that $\vc(\C) = d+1 > 0$.

  \begin{claim}
    Let $k_0 := 2^{d+1}d+1$.
    There is $D \subseteq  A$ and $c' \in \C|_D$ which is $k_0$-compressible in $\C|_D$ such that $\vc(\C_{c'}) \leq  d$.
  \end{claim}
  \begin{proof}
    We may assume $|A| \geq  k_0$, as otherwise the result is immediate (take $D=A$ and any $c'\in \C$).

    Let $\S := \set{c' \in \C|_D}{D \subseteq  A,\; c' \text{\;is\;}k_0\text{-compressible in\;} \C|_D}$.
    Equip $\S$ with the partial order of inclusion, i.e., $\C|_{D_1} \ni c'_1 \leq  c'_2 \in \C|_{D_2}$ iff $D_1 \subseteq  D_2$ and $c'_2|_{D_2} = c'_1$.

    Then $\S$ is closed under unions of chains.
    Indeed, if $(c'_i)_{i\in I}$ is a chain with $c'_i \in \C|_{D_i}$, then the sets of extensions to $\C$, $\{ c' \in \C : c'|_{D_i} = c'_i \}$, form a chain $\mathcal{F}$ of closed non-empty subsets of $\C$; but $\C$ is closed in $2^A$ hence compact, so $\mathcal{F}$ has non-empty intersection. Hence $\bigcup_{i \in I} c'_i \in \C|_{\bigcup_{i \in I} D_i}$. Meanwhile, $\bigcup_{i \in I} c'_i$ is $k_0$-compressible since each $c'_i$ is.

   So by Zorn's lemma, $\S$ has a maximal element $c' \in \C_D$.

    We conclude by showing that $\vc(\C_{c'}) \leq  d$.

    Otherwise, $2^B \subseteq  \C_{c'}$ for some $B \subseteq  A$ with $|B| = d+1$. Note that $B \cap D = \emptyset$.
    We claim that there is $e \in 2^B$ such that $c' \cup e \in \S$, contradicting maximality of $c'$.

    Indeed, let $A_0 \subseteq_{\fin}  D\cup B$ and let $D_0 := A_0 \cap D$.
    Then there is $D_1 \subseteq  D$ with $|D_1| \leq  k_0$ such that $c'|_{D_1} \vdash _{\C_{D}} c'| _{D_0}$.
    In fact we may take $D_1$ with $|D_1| = k_0$, since $|D| \geq  k_0$ by maximality and the assumption that $|A| \geq  k_0$. 
    Since $\vc(\C) \leq  d+1 = |B|$ and $2^B \subseteq  \C_{c'}$, for each $a \in D_1$ there is $e_a \in 2^B$ such that $e_a \vdash _{\C} c'|_{\{a\}}$.
    By the choice of $k_0=2^{d+1}d+1$ and the pigeonhole principle, there exist $e \in 2^B$ and $E \subseteq  D_1$ such that $|E| = d+1$ and $e \vdash _{\C} c'|_E$.
    Let $A_1 := (D_1 \cup B) \setminus E$, so $|A_1| = |D_1| = k_0$.
    Then
    \begin{equation} \label{eqn:the point of e}
      (c' \cup e)|_{A_1} \vdash _\C (c' \cup e)|_{A_0} .
    \end{equation}

    In this way we obtain a $2^B$-colouring of the partial order of finite subsets of $D \cup B$, where each finite $A_0 \subseteq  D\cup B$ is coloured with an $e \in 2^B$ such that (\ref{eqn:the point of e}) holds for some $A_1 \subseteq  D\cup B$ with $|A_1| = k_0$. By \cref{rem:cofinal finite colouring} there is a cofinal monochromatic subset, yielding $e \in 2^B$ which is as required.
  \end{proof}

  Now by the induction hypothesis there is $c \in \C_{c'}$ which is $\kc(d)$-compressible in $\C_{c'}$.
  We conclude by showing that $c$ is $(\kc(d)+k_0)$-compressible in $\C$;
  this gives the stated bound, since $(2^{d+1}(d-2)+d+4) + (2^{d+1}d + 1) = 2^{(d+1)+1}((d+1)-2)+(d+1)+4$.

  So suppose $A_0 \subseteq_{\fin}  A$ and let $A_1 \subseteq  A$ be such that $|A_1| \leq  \kc(d)$ and $c|_{A_1} \vdash _{\C_{c'}} c|_{A_0}$.
  By compactness of $\C$ it follows that there is a finite subset $D_0 \subseteq  D$ such that $c|_{A_1} \vdash _{\C_{c'|_{D_0}}} c|_{A_0}$.

  Let $D_1$ be such that $|D_1| \leq  k_0$ and $c'|_{D_1} \vdash _{\C|_D} c'|_{D_0}$ (which exists as $c' \in \S$).
  Then $c|_{A_1\cup D_1} \vdash _{\C} c|_{A_0}$, as required.
\end{proof}

\begin{remark}
  For finite $A$, the exponential dependency of $\kc(d)$ on $d$ obtained in \cite{CCT} was improved to a quadratic dependency in \cite{k-isolationQuadratic}. Conjecturally it is even linear (see the introduction to \cite{k-isolationQuadratic}). The proof of this quadratic bound does not adapt so readily to the infinite case, and it would be interesting to find the best bound, and in particular to see whether \cref{thm:infinite CCT compression} holds with a quadratic bound.
\end{remark}

\subsection{Density of compressible local types}
\label{sec:local compression}

In the following definition, we use the notation $p\vdash_{\pi} q$ for $p \cup \pi \vdash q$, where $p,\pi$ are small partial types and $q$ a finite partial type (this is compatible with the notation in \cref{def:compressible Boolean functions} when $p,q$ are complete $\phi$-types, and $\C$ is the set of $\phi$-types consistent with $\pi$). As usual, we work in a complete $\L$-theory $T$.

\newcommand{\down}{\downarrow}
\begin{definition} \label{def:local compression}
  Fix a formula $\phi(x,y)$, a parameter set $A \subseteq  \U^y$ and a small partial type $\pi(x)$. Recall the notation $S_\phi^\pi(A)$ from \cref{sec:notations}.
  \begin{itemize}
    \item $p \in S_\phi^\pi(A)$ is \defn{$k$-compressible modulo $\pi$} if it is compressible in $S_\phi^\pi(A)$ considered as a (closed) subspace of $2^A$ as in \cref{def:compressible Boolean functions}: for any finite $A_0 \subseteq  A$ there is $A_1 \subseteq  A$ with $|A_1| \leq  k$ such that $p|_{A_1}(x) \vdash_{\pi} p|_{A_0}(x)$.

    \item $p \in S_\phi^\pi(A)$ is \defn{$\star$-compressible modulo $\pi$} if it is $k$-compressible modulo $\pi$ for some $k<\omega$.

    \item $p \in S_\phi(A)$ is \defn{$k$-} resp.\ \defn{$\star$-compressible} if it is $k$- resp.\ $\star$-compressible modulo $x=x$.
    \item $S_{\phi \down k}^\pi(A) \subseteq  S_\phi^\pi(A)$ (respectively $S_{\phi\down \star}^\pi(A)$) is the space of $k$-compressible (respectively $\star$-compressible) $\phi$-types modulo $\pi$. When $A \subseteq \U$, $S_{\phi \down k}^\pi(A) = S_{\phi \down k}^\pi(A^y)$.
  \end{itemize}
\end{definition}

\begin{remark}
  In terms of Borel complexity, if $A$ is countable then $S_{\phi \down \star}^\pi(A)$ is a $\mathbf{\Sigma^0_3}$-subset of $S_\phi^\pi(A)$: it is a countable union (going over all $k$) of countable intersections (going over all finite subsets of $A$) of countable unions (going over all subsets of $A$ of size $\leq k$) of clopen sets (the implication).
\end{remark}

\begin{corollary} \label{cor:localComp}
  Let $\phi(x,y)$ be a formula, $d \in \N$, $A \subseteq  \U^y$ and $\pi(x)$ a small partial type. Suppose that $\phi(x,y)$ is NIP and that $\vc^*(\phi) \leq d$. Let $k=\kc(d)$ be as in \cref{thm:infinite CCT compression}.
  \begin{enumerate}[(i)]\item $S_{\phi\down k}^\pi(A) \neq  \emptyset $.
  \item If $A' \subseteq  A$, then any $p' \in S_{\phi\down l}^\pi(A')$
  extends to some $p \in S_{\phi\down (l+k)}^\pi(A)$.
  \item $S_{\phi\down\star}^\pi(A)$ is dense in $S_\phi^\pi(A)$.
  \end{enumerate}
\end{corollary}
\begin{proof}
  \begin{enumerate}[(i)]
    \item This is immediate from \cref{thm:infinite CCT compression} by identifying $\C$ with $S^{\pi}_{\phi}(A)$.
    \item By (i), there is some $p \in S_{\phi\down k}^{\pi\cup p'}(A)$.
    Then if $A_0 \subseteq_{\fin}  A$, there is $A_1 \subseteq  A$ with $|A_1| \leq  k$ such that $p' \cup p|_{A_1} \vdash _\pi p|_{A_0}$. By compactness, there is a finite $A_0' \subseteq  A'$ such that $p'|_{A_0'} \cup p|_{A_1} \vdash _\pi p|_{A_0}$, and then by $l$-compressibility modulo $\pi$ of $p'$ we have $p'|_{A_1'} \vdash _\pi p'|_{A_0'}$ for some $A_1' \subseteq  A'$ with $|A_1'| \leq  l$. Then $p|_{A_1' \cup A_1} = p'|_{A_1'} \cup p|_{A_1} \vdash _\pi p|_{A_0}$. So $p$ is $(l+k)$-compressible modulo $\pi$.
    \item A basic open subset of $S_\phi^\pi(A)$ is of the form
    $S_\phi^{\pi\cup p'}(A)$ where $p' \in S_\phi^\pi(A')$ and $A' \subseteq_{\fin} A$.
    Clearly $p' \in S_{\phi\down|A'|}^\pi(A')$, so by (ii) there is $p \in S_{\phi\down(|A'|+k)}^\pi(A)$ extending $p'$.
  \end{enumerate}
\end{proof}

\begin{corollary} \label{cor:NIP iff k-compressible exists}
  The following are equivalent for a formula $\phi(x,y)$ and a partial type $\pi(x)$.
  \begin{enumerate}
    \item For some $\psi(x)$ such that $\pi \vdash \psi$, $\psi(x)\land \phi(x,y)$ is NIP.
    \item There exists $k<\omega$ such that for any set $A \subseteq \U^y$,  $S_{\phi\down k}^\pi(A) \neq  \emptyset$.
    \item For any set $A \subseteq \U^y$, there exists $k<\omega$ such that $S_{\phi\down k}^\pi(A) \neq  \emptyset$.
  \end{enumerate}
\end{corollary}

\begin{proof}
  (1) implies (2) is \cref{cor:localComp}(i) (any type in $S_{{\psi \land \phi}\down k}^\pi(A)$ naturally induces one in $S_{\phi\down k}^\pi(A)$) and (2) implies (3) is clear.

  $\neg$(1) implies $\neg$(3). By $\neg$(1), $\U^y$ is infinite. By compactness there is $A:=\set{a_i}{i<\omega} \subseteq \U^y$ such that for any $A' \subseteq A$, there is some $b_{A'} \vDash \pi$ such that $\phi(b_{A'},a)$ holds iff $a \in A'$ for any $a \in A$. Suppose $p \in  S_{\phi\down k}^\pi(A)$ for some $k <\omega$. Then for some $A_0 \subseteq A$ of size $\leq k$, $p|_{A_0} \vdash_\pi p|_{a_{\leq k}}$. But setting $A' := \{ a \in A_0 : \phi(x,a) \in p \} \cup \{ a \in A \setminus A_0 : \neg\phi(x,a) \in p \}$, we have $b_{A'} \vDash p|_{A_0}$; but $b_{A'} \not \vDash p|_a$ for any $a \in A \setminus A_0$, and $a_{\leq k} \not\subseteq A_0$ since $|A_0| \leq k$.
\end{proof}

This gives a new characterisation of NIP types.

\begin{definition}\label{def:NIP type}
  We say that a partial type $\pi(x)$ has \defn{IP} if there is a formula $\phi(x,y)\in\L$ which has IP as witnessed by realisations of $\pi$, i.e., if $\vc(\{ \phi(\pi(\U), a) : a \in \U^y\}) = \infty$.
  A formula or a partial type is \defn{NIP} if it does not have IP.
\end{definition}

By compactness we have that:
\begin{remark}\label{rem:NIP type in term of formula}
  A partial type $\pi(x)$ is NIP iff for every formula $\phi(x,y)$ there is a formula $\psi(x)$ implied by $\pi$ such that $\psi(x)\land \phi(x,y)$ is NIP (as a formula over $\U$).
\end{remark}

By \cref{cor:NIP iff k-compressible exists} and \cref{rem:NIP type in term of formula} we have:
\begin{corollary}
  A partial type $\pi(x)$ is NIP iff for every formula $\phi(x,y)$ and $A \subseteq \U^y$ there is $k<\omega$ such that $S_{\phi\down k}^\pi(A) \neq  \emptyset$.
\end{corollary}

\subsection{Density of compressible types in countable NIP theories} \label{sec:density of compressible in ctbl NIP}

Now we turn from local types to types.

\begin{definition} \label{def:global compression}
  Let $\pi(x) \subseteq  \pi'(x)$ be partial types over a parameter set $A \subseteq \U$ (perhaps contained in a collection of sorts).

  \begin{itemize}
    \item A formula $\zeta(x,z)$ \defn{compresses $\pi$ within $\pi'$ with respect to $A$} if for any finite $A_0 \subseteq  A$ there exists $a \in A^{z}$ such that
    \[ \pi'(x) \vdash  \zeta(x,a) \vdash  \pi|_{A_0}(x) .\]
    If $\pi'$ is clear from the context we omit it.

    \item $\pi$ is \defn{compressible within $\pi'$ with respect to $A$} if some $\zeta$ compresses $\pi$ within $\pi'$ with respect to $A$.

    \item $\pi$ is \defn{t-compressible}\footnote{The letter 't' stands for totally, thoroughly, or typewise.} with respect to $A$ if $\pi$ is compressible within $\pi$ with respect to $A$.
  \end{itemize}

  Let $p(x) \in S(A)$.
  \begin{enumerate}
    \item $p$ is \defn{compressible} if for each formula $\phi(x,y)$, the restriction $p \restriction \phi \in S_\phi(A)$ of $p$ to a $\phi$-type is compressible within $p$ with respect to $A$.

    \item $p$ is \defn{strongly compressible} if for each formula $\phi$  there exists a finite set of formulas $\Delta \ni \phi$ such that $p \restriction \Delta$ is t-compressible with respect to $A$.
  \end{enumerate}
\end{definition}

\begin{remark}
  Note that the definition above of a compressible type is the same as \cref{def:compressible types prelim}.
\end{remark}

\begin{remark}
  The reason we say ``with respect to $A$'' in the definition is because a partial type over $A$ is also a partial type over any set containing $A$. In the future we will usually omit this since $A$ will be clear from the context.
\end{remark}

\begin{remark} \label{rem:compressible iff finite restrictions}
  As we said in \cref{sec:notations}, we do not restrict ourselves to finitary types. Note that $p\in S^x(A)$ is compressible iff all of its restrictions to finite tuples of variables are compressible. 
\end{remark}

\begin{remark}
  The relations between these definitions and the definitions for $\phi$-types in \cref{def:local compression} are slightly subtle. In particular, for $A \subseteq \U^y$ and a $\phi$-type $p \in S_\phi(A)$, the condition that $p$ is $\star$-compressible (i.e., $k$-compressible for some $k$) is strictly stronger than the condition that $p$ is t-compressible. For example, in $\Th(\N;<)$, the non-realised $(x=y)$-type in $S_{x=y}(\N)$ is t-compressed by $x>z$, but is not $k$-compressible for any $k<\omega$.
\end{remark}

\begin{remark} \label{rem:passing to eq}
  Note that for a model $M$ and $p \in S(M)$, $p$ is (strongly) compressible iff its unique extension $p^{\eq}$ to $M^{\eq}$ is (strongly) compressible (by translating formulas in $\L^{\eq}$ to formulas in $\L$, see \cite[Lemma~1.1.4]{PillayGeometric}).
\end{remark}

\begin{lemma} \label{lem:adding one more formula}
  Let $\pi(x)$ be a t-compressible partial type over a set $A \subseteq \U^y$,
  and let $\phi(x,y)$ be an NIP formula.
  Then there exists $p_\phi \in S_\phi(A)$ such that $\pi \cup p_\phi$ is consistent and t-compressible.

  Moreover, there is a formula $\xi(x,w)$ which is a Boolean combination (depending only on $\vc^*(\phi)$) of instances of $\phi$ and equality such that if $\zeta(x,w')$ t-compresses $\pi$ (i.e., compresses $\pi$ within itself) then $\zeta(x,w') \wedge \xi(x,w)$ t-compresses $\pi \cup p_\phi$.
\end{lemma}
\begin{proof}
  We may assume $|A| > 1$, as otherwise the result is clear.

  By \cref{cor:localComp}(i), there is $p_\phi \in S_{\phi\down k}^\pi(A)$ for some $k<\omega$ depending only on $\vc^*(\phi)$.
  By a coding of finitely many formulas as one as in the proofs of e.g.\ \cite[Theorem~II.2.12(1)]{Sh-CT} and \cite[Lemma~2.5]{MR2963018}, we obtain a formula $\xi(x,w)$ such that for any finite $A_0 \subseteq A$, there is $a \in A^{w}$ such that $p_\phi(x) \vdash  \xi(x,a)$ and $\pi \cup \{\xi(x,a)\} \vdash  p_\phi|_{A_0}$.
  Explicitly, we may take $\xi(x,w)$ with $w=\sequence{w^i_j}{i<3,j<k}$ to be
  $\bigwedge_{j<k} ( \phi(x,w^0_j) \leftrightarrow  w^1_j = w^2_j )$. Then, for any finite $A_0$, there is some $\sequence{\epsilon_j}{j<k} \in 2^k$ and $\sequence{c_j}{j<k} \in A^k$ such that $\bigwedge_{j<k} \phi(x,c_j)^{\epsilon_j} \vdash p|_{A_0}$. Let $d_0 \neq d_1 \in A$. For $j<k$, let $a^0_j = c_j$, $a^1_j = d_0$, and let $a^2_j = d_0$ if $\epsilon_i = 1$ and otherwise let $a^2_j = d_1$. Finally, let $a = \sequence{a^i_j}{i<3,j<k}$.

  Now assume that $\zeta(x,w')$ is as in the lemma and fix some finite set $A_0 \subseteq A$. Let $a \in A^{w}$ be as above. By compactness there is a finite $A_0' \subseteq  A$ such that $A_0 \subseteq A_0'$ and $\pi|_{A_0'} \cup \{\xi(x,a)\} \vdash  p_\phi|_{A_0}$,
  and so (by the assumption on $\zeta$) there is $a' \in A^{w'}$ such that $(\pi \cup p_\phi)(x) \vdash  (\zeta(x,a') \wedge \xi(x,a)) \vdash  (\pi \cup p_\phi)|_{A_0}$.
\end{proof}

\begin{corollary} \label{cor:density of comp in countable NIP}
  ($T$ countable NIP)
  Suppose $A \subseteq \U$ is a set of parameters and $x$ is a countable tuple of variables. Then, compressible types are dense in $S^x(A)$:

  If $\theta(x)$ is a consistent formula over $A$,
  then there exists a compressible type $p(x) \in S(A)$ with $p(x) \vdash  \theta(x)$.

  More generally, if $\pi(x)$ is a t-compressible partial type over $A$, then there exists a strongly compressible $p \in S(A)$ with $\pi \subseteq  p$.
\end{corollary}
\begin{proof}
  Clearly it is enough to prove the ``more generally'' part, so assume $\pi$ is t-compressible and $\zeta$ compresses $\pi$ within $\pi$.

  Enumerate the formulas $\phi(x,y)$ as $\sequence{\phi_i(x,y_i)}{i<\omega}$ (where the $y_i$'s are finite), with $\phi_0 = \zeta$. For $i<\omega$, let $\Delta_i = \set{\phi_j}{j<i} \cup \{x=y\}$.
  Let $\pi_0 = \pi$.
  Recursively applying \cref{lem:adding one more formula}, let $p_{\phi_i} \in S_{\phi_i}(A)$ be such that $\pi_{i+1} := \pi_i \cup p_{\phi_i}$ is t-compressible, and moreover is compressed by a Boolean combination of formulas from $\Delta_{i+1}$.
  Then each $\pi_i\restriction \Delta_i$ is t-compressible,
  and so $p := \bigcup_{i<\omega} \pi_i$ is strongly compressible.
\end{proof}

For an example showing the necessity of the countability assumption, see \cref{rem:countability necessary} below.

\begin{remark} \label{rem:converses}
  It follows from \cref{cor:NIP iff k-compressible exists} that \cref{lem:adding one more formula} characterises $\phi$ being NIP (letting $\pi$ be the empty type). However, \cref{cor:density of comp in countable NIP} does not characterise NIP for countable theories. An easy example is $\Th(\N,+,\cdot)$, and in fact any theory with IP in which $\dcl(A)$ is a model for any set $A$ (given a consistent formula $\theta(x)$ over a set $A$, let $c \vDash \theta$ be in $\dcl(A)$, then $\tp(c/A)$ is compressible and even isolated).
\end{remark}

\begin{question}\label{que:weak compressibility}
  We could consider an apparently weaker notion of compressibility of a type: say $p \in S(A)$ is \defn{weakly compressible} if for any formula $\phi(x,y)$ there is some formula $\zeta(x,z)$ such that for any finite $A_0 \subseteq A$ there is some $d \in \U^z$ such that $p \vdash \zeta(x,d)$ and $\zeta(x,d)\vdash (p\restriction\phi)|_{A_0}$. Note that if the base $A$ is a model, then weak compressibility is equivalent to compressibility, but for general sets it is less clear. In \cref{exa:transitivity is false for ABA} below we will see that if $T$ is the theory of atomless Boolean algebras, this can fail. Is it true that if $T$ is NIP then $p$ is weakly compressible iff $p$ is compressible?


\end{question}

\section{Compressibility and stability} \label{sec:comp and stability}

Here we discuss compressibility in the context of stability, in both the local and global senses, and point out that compressibility is equivalent to l-isolation (see \cref{def:l-isolation}) in these contexts. The main results are:
\begin{itemize}
  \item For stable formulas, $k$-compressibility is equivalent to $k$-isolation (\cref{lem:stable Comp = Isol Local}).
  \item For stable types, compressibility is equivalent to l-isolation (\cref{lem:comp stable type}), and in particular when $T$ is stable these two notions are the same.
  \item For generically stable types, compressibility is equivalent to l-isolation (\cref{prop:generically stable compressible types are l-isolated}).
\end{itemize}

\subsection{Stable formulas}

Recall that a formula $\phi(x,y)$ is \defn{stable} if it does not have the order property: there are no $\sequence{a_i,b_i}{i<\omega}$ such that $\phi(a_i,b_j)$ holds iff $i<j$, and $\phi$ has \defn{the strict order property} (\defn{SOP}) if there is a sequence $\sequence{b_i}{i<\omega}$ such that $\sequence{\phi(\U^x,b_i)}{i<\omega}$ forms a strictly decreasing sequence of definable sets (with respect to containment). A theory $T$ is \defn{stable} if no formula has the order property. Clearly if $\phi$ is stable, it is NIP.

\begin{definition}
  Suppose $k<\omega$ and $p\in S_\phi(A)$ for some $A \subseteq \U^y$. Then $p$ is \defn{$k$-isolated} if for some $A_0 \subseteq A$ such that $|A_0|\leq k$,  $p|_{A_0} \vdash  p$.  $p$ is \defn{isolated} if it is $k$-isolated for some $k$ (this coincides with the usual topological definition).
\end{definition}

\begin{remark}\label{rem:k-comp + isolated => k-isol}
  Note that if $p \in S_\phi(A)$ is $k$-isolated then it is $k$-compressible. Also, if $p$ is $k$-compressible and isolated, then $p$ is $k$-isolated.
\end{remark}

The following says in particular that under stability, $k$-compressibility and $k$-isolation are the same.

\begin{lemma} \label{lem:stable Comp = Isol Local}
  Let $\phi(x,y)$ be NIP. Then the following are equivalent:
  \begin{enumerate}[(i)]
    \item $\phi$ is stable.
    \item For any $B \subseteq  \U^y$ and $p \in S_\phi(B)$ and $k \in \N$, if $p$ is $k$-compressible then $p$ is isolated (and hence $k$-isolated by \cref{rem:k-comp + isolated => k-isol}).
    \item For all $k<\omega $ and $\ee \in \{0,1\}^k$,
    \[\theta_{\ee}(x,\yy) := \bigwedge_{j<k} \phi(x,y_j)^{\eps_j}\]
    does not have the strict order property.
  \end{enumerate}
\end{lemma}
\begin{proof}

  (i) implies (iii) as a Boolean combination of stable formulas is stable (see e.g., \cite[Lemma~2.1]{PillayGeometric}) and the strict order property implies the order property.

  $\neg$(i) implies  $\neg$(iii) by the proof of \cite[Theorem~2.67]{simon-NIP} and the subsequent remark.

  $\neg$(ii) implies $\neg$(iii): let $p \in S_\phi(B)$ be $k$-compressible but not $k$-isolated.

  Inductively we find $\bb^i \in B^k$ and $\ee^i \in \{0,1\}^k$ for $i < \omega$ such that we have $p(x) \vdash  \theta_{\ee^i}(x,\bb^i)$ for all $i$,
  and $\theta_{\ee^i}(x,\bb^i) \vdash  \theta_{\ee^j}(x,\bb^j)$ iff $i \geq  j$.
  (Given $i$, since $\theta_{\ee^i}(x,\bb^i) \not \vdash p$, there is $b \in B$ such that $\theta_{\ee^i}(x,\bb^i) \not \vdash p|_b$. Let $\theta_{\ee^{i+1}}(x,\bb^{i+1}) \vdash p|_{\bb^ib}$ be from $p$.)
  But then some $\ee$ occurs infinitely often, and then $\theta_{\ee}(x,\yy)$ has SOP.

  $\neg$(iii) implies  $\neg$(ii):
  suppose $\bb^i \in \U^{\yy}$ for $i<\omega$ and
  $\theta_{\ee}(x,\bb^i) \vdash \theta_{\ee}(x,\bb^j)$ iff $i \geq  j$.
  Let $B = \set{b^i_j}{i<\omega,\; j<k} \subseteq  \U^y$, and note that $\set{\theta_{\ee}(x,\bb^j)}{j<\omega}$ implies a complete type $p \in S_\phi(B)$.
  Then $p$ is $k$-compressible but not isolated.
\end{proof}

\subsection{Stable types and theories}

\begin{definition} \label{def:stable type}
  A partial type $\pi(x)$ over $A$ is \defn{stable} if every extension $p\in S(B)$, over every $B\supseteq A$, is definable.
\end{definition}

It is well-known that $T$ is stable if and only if every type is definable (see e.g., \cite[Corollary 8.3.2]{TentZiegler}), so $T$ is stable if and only if every partial type is stable. For more on stable types (including equivalent definitions), see \cite{Gen,MR4216280}. We will use the following equivalence:

\begin{fact}\cite[Remark~2.6]{MR4216280}\label{fac:stable types}
 The following are equivalent for a partial type $\pi(x)$:
 \begin{enumerate}
  \item $\pi$ is stable.
  \item For every formula $\phi(x,y)$ there is a formula $\psi(x)$ implied by $\pi$ such that $\phi(x,y) \land \psi(x)$ is stable (as a formula over $\U$).
 \end{enumerate}
 \end{fact}

Under stability, the analogue of compressibility of a type is l-isolation.

 \begin{definition}\label{def:l-isolation}
  A type $p(x) \in S(A)$ is \defn{l-isolated} if for each formula $\phi(x,y)$ there is $\zeta \in p$ with $\zeta \vdash  p \restriction \phi$.
\end{definition}

Clearly, an l-isolated type is compressible. By considering the formula $x\neq y$, we easily obtain:

\begin{remark} \label{rem:l-isol is realised}
  Any l-isolated type over a model is realised.
\end{remark}

The following is analogous to (but not actually comparable with) \cref{lem:stable Comp = Isol Local}.

\begin{lemma} \label{lem:comp stable type}
  \begin{enumerate}[(i)]\item
  Suppose $p \in S(A)$ is compressible but not l-isolated. Then:
  \begin{enumerate}
    \item There are tuples $a_i,b_i$ in $A$ which witness the order property for some $\L$-formula.
    \item $p$ is not stable.
  \end{enumerate}
  In particular, if $T$ is stable then any compressible type is l-isolated.
  \item ($T$ countable NIP) $T$ is stable iff any compressible type is l-isolated, iff there is some $\omega$-saturated model $M$ such that every strongly compressible type over $M$ is l-isolated.
  \end{enumerate}
\end{lemma}
\begin{proof}

  For both (i.a) and (i.b), suppose $\phi(x,y)$ witnesses that $p$ is not l-isolated and $\zeta(x,z)$ compresses $p \restriction \phi$.

  (i.a) Let $\theta(y,z) = \bigvee_{\epsilon<2} \forall x (\zeta(x,z) \rightarrow  \phi(x,y)^\epsilon)$. We recursively construct $a_i \in A^y$ and $b_i \in A^z$ for $i < \omega$, such that $\vDash  \theta(a_i,b_j) \Leftrightarrow  i<j$:
  if $a_{<i}$ and $b_{<i}$ are already defined, let $b_i$ be such that $p \ni \zeta(x,b_i) \vdash  (p \restriction \phi)|_{a_{<i}}(x)$, and let $a_i$ be such that $\zeta(x,b_j) \not\vdash  (p \restriction \phi)|_{a_i}(x)$ for all $j\leq i$, which exists since $\bigwedge_{j\leq i} \zeta(x,b_j) \not\vdash  (p \restriction \phi)(x)$.

  (i.b) For $\psi(x) \in p$, we show that $\theta(x,z):=\zeta(x,z)\land \psi(x)$ has the order property by recursively constructing $\sequence{a_i, b_i, c_i}{i<\omega}$ such that  $\vDash \theta(a_i,b_j)$ iff $i \geq j$ and $\bigwedge_{j < i} \theta(x,b_j) \in p$ and $\vDash \phi(a_i,c_i) \Leftrightarrow \phi(x,c_i) \notin p$. This is enough by \cref{fac:stable types}. Suppose we found $\sequence{a_j,b_j,c_j}{j<i}$. Let $b_i$ be such that $p \ni \zeta(x,b_i) \vdash  (p \restriction \phi)|_{c_{<i}}(x)$. Since $\bigwedge_{j \leq i}\theta(x,b_j)$ does not isolate $p\restriction \phi$, there are some $a_i,c_i$ such that $\vDash \bigwedge_{j \leq i} \theta(a_i,b_j)$, $\phi(x,c_i)^{\epsilon} \in p$ for some $\epsilon<2$, and $\neg \phi(a_i,c_i)^{\epsilon}$ holds.

  (ii) The implications from left to right follow by (i) and trivially, respectively. For the other direction, assume that $T$ is not stable.
  By \cite[Theorem~2.67]{simon-NIP}, $T$ has the SOP.
  So say $<$ is an $\emptyset$-definable (strict) preorder on an $\emptyset$-definable set $D$ with infinite chains.

  Let $M$ be an $\omega$-saturated model. So $M$ contains an infinite chain $C$ which we may assume is maximal. Since $C$ is infinite, we can write $C = C_1 + C_2$ where either $C_1$ has no last element or $C_2$ has no first element (one of them may be empty).

  Let $\pi(x)$ be the unique type in $<$ over $C$ corresponding to the cut $(C_1,C_2)$ (i.e., $\pi(x)$ is determined by $\set{c_1<x}{c_1 \in C_1}\cup \set{x<c_2}{c_2 \in C_2}$). 
  Now, $\pi$ is t-compressible (e.g., if both $C_1,C_2$ are nonempty, then $z<x<w$ compresses $\pi$ within $\pi$ and if $C_1$ is empty then $x<z$ compresses $\pi$ within $\pi$). Hence by \cref{cor:density of comp in countable NIP}, $\pi$ has a strongly compressible completion $q \in S(M)$.
  By maximality of $C$, $\pi$ is not realised in $M$, so neither is $q$.
  So by \cref{{rem:l-isol is realised}}, $q$ is not l-isolated.

  (Note that we could have worked with a partial order instead of a preorder by passing to eq and using \cref{rem:passing to eq}.) 
\end{proof}

\begin{remark}\label{rem:enough infinite chain}
  From the proof of \cref{lem:comp stable type}(ii) it follows that if $M \vDash T$ contains an infinite chain in an $M$-definable preorder $D$ then there is $c\in D$ such that $\tp(c/M)$ is compressible but not l-isolated.
\end{remark}

\begin{remark}
  From \cref{cor:density of comp in countable NIP}, it follows that l-isolated types are dense in stable theories, but this was well-known and follows easily by 2-rank considerations, see \cite[Lemma~IV.2.18(4)]{Sh-CT}.
\end{remark}

\begin{remark} \label{rem:countability necessary}
  The following example demonstrates the necessity of the countability assumption on $T$ in \cref{cor:density of comp in countable NIP} even for stable theories.

  Let $\kappa$ be a cardinal, and consider $\kappa$ colourings on a set $X$, with each colouring using the same colours, such that no point gets the same colour according to two different colourings, but apart from this restriction all possibilities are realised. We can formalise this in the language with a sort $X$, a sort $C$ for the colours, and for each $i \in \kappa$ a function $f_i : X \rightarrow  C$ giving the colour of an element according to the $i$-th colouring. The theory is axiomatised by saying there are infinitely many colours and, for each finite set $\{i_1,\ldots ,i_n\} \subseteq  \kappa$ enumerated without repetitions and each $m \geq  0$, an axiom
  \begin{align*} \forall c_1,\ldots ,c_n \in C\; \forall x_1,\ldots ,x_m \in X\;
  (& \exists x \in X\; ( \bigwedge_{j=1}^n f_{i_j}(x) = c_j \wedge \bigwedge_{j=1}^m x \neq  x_j ) \\
  &\leftrightarrow  \bigwedge_{j \neq  k} c_j \neq  c_k ) \end{align*}

  This axiomatises a complete consistent theory $T$ with quantifier elimination in the given language. Indeed, restricting to any finite sublanguage $\L_0$ containing $X,C$ and finitely many function symbols $f_i$, $T \restriction \L_0$ is the Fra\"iss\'e limit of the class of finite structures $(X_0,C_0)$ where for every $f_i, f_j \in \L_0$ and all $x\in X_0$, if $f_i(x)=f_j(x)$ then $i=j$. It follows that $T$ is stable; indeed, $|S^1(A)| \leq  \lambda^\kappa$ for $|A| \leq  \lambda$, so $T$ is $2^{\kappa}$-stable.

  Now suppose $\kappa > \aleph_0$ and let $C_0 \subseteq  C$ with $|C_0| = \aleph_0$. We claim that the formula $x \in X$ has no compressible (equivalently, by \cref{lem:comp stable type}(i), l-isolated) completion $p \in S^x(C_0)$. Indeed, it is easy to see that $p$ would have to include for each $i$ a formula $f_i(x)=c_i$ for some $c_i \in C_0$, but then $c_i \neq  c_j$ for $i \neq  j$, contradicting $\kappa > |C_0|$.

  Other (hints for) examples are given in \cite[Exercise~IV.2.13]{Sh-CT} where Shelah also gives a \emph{superstable}\footnote{Recall that $T$ is \defn{superstable} if it is stable in all cardinals $\geq 2^{|T|}$.} counterexample, which we will describe briefly. Let $\L = \set{E_\nu,P_s}{\nu \in \omega^\omega,s\in \omega^{<\omega}}$ where the $P_s$'s are unary predicates and the $E_\nu$'s are binary relation symbols. Let $M$ be the $\L$-structure whose universe is $\omega^\omega \times \omega$ where $P_s^M = \set{(\eta,n) \in \omega^\omega\times \omega}{s \triangleleft \eta}$ and $E_\nu^M$ is an equivalence relation where $(\eta_1,n_1),(\eta_2,n_2)$ are equivalent iff ($\eta_1 = \eta_2$ or for some $n<\omega$, $\eta_1 \restriction n = \eta_2 \restriction n = \nu \restriction n$ and $\eta_1(n) = \eta_2(n) \neq \nu(n)$). Essentially, each class is infinite and two branches in the tree $\omega^\omega$ are $\nu$-equivalent if they divert from $\nu$ at the same point and in the same direction (starting the same cone), and $\nu$ is $\nu$-equivalent only to itself. Let $T = \Th(M)$. It is not too hard to see that $T$ has quantifier elimination.
  We leave it as exercise to check that for any set $A$, $S^1(A) \leq |A|+2^{\aleph_0}$, and thus $T$ is superstable. 

  Finally, working in $\U^{\eq}$ and letting $A = \set{(\eta,n)/E_\nu}{\eta \neq \nu, n<\omega}$, there is no l-isolated type $p\in S^1(A)$ (in the home sort). Indeed, if $p(x)$ is such a type, then for every $\nu \in \omega^\omega$ there is some $\eta \neq \nu$ such that $x/E_\nu = (\eta,0)/E_\nu$ is in $p$.
  It follows that for some $\nu \in \omega^\omega$, $p$ must contain $\set{P_s(x)}{s \triangleleft \nu}$. 
  But then $p \vdash x/E_\nu \neq (\eta,0)/E_\nu$ for all $\eta \neq \nu$, contradiction.

\end{remark}
\subsection{Generically stable types}

Generically stable types are invariant types that exhibit stability-like behavior ``generically'' i.e., when considering their Morley sequences. This notion was first studied in the NIP context by Shelah \cite{MR2062198} (under the name ``stable types'') and then by Hrushovski and Pillay \cite{MR2800483} and independently Usvyatsov \cite{MR2499428}. See also \citep[Section~2.2.2]{simon-NIP}. It was defined in general in \cite{AnandPredrag} by Pillay and Tanovi\'{c}. See also \cite{MR4033642} for more on generic stability outside of the NIP context.

\begin{definition} \label{def:GS global type}
  We say that a global type $p$ is \defn{generically stable over} $A$ if it is $A$-invariant and for every ordinal $\alpha$, every $\phi(x)$ with parameters in $\U$ and every Morley
  sequence $\sequence{a_i}{i<\alpha}$ of $p$ over $A$, the set $\set{i<\alpha}{\U \vDash\phi(a_i)}$ is finite or cofinite.
\end{definition}

\begin{fact} \label{fac:generically stable -> dfs}
  Suppose $p$ is generically stable over $A$. Then:
 \begin{enumerate}
    \item \cite[Proposition~2.1]{AnandPredrag} $p$ is $A$-definable and finitely satisfiable in every model containing $A$.
    \item \cite[Proposition~2.1]{AnandPredrag} If $I$ is a Morley sequence of $p$ over $A$ then $I$ is an indiscernible set (i.e., totally indiscernible).
    \item \cite[Proposition~3.2]{MR4033642} $p = \lim\sequence{a_i}{i<\omega}$ for any Morley sequence $\sequence{a_i}{i<\omega}$ of $p$ over $A$, i.e., $p$ is the limit type of any of its Morley sequences over $A$: $\theta(x) \in p$ iff $\theta(a_i)$ holds for all but finitely many $i<\omega$. (In \cite[Proposition~3.2]{MR4033642} it is stated over models, but it is also true over sets and follows directly from the definition; we leave this to the reader.) 
  \end{enumerate}

  Moreover, by \cite[Theorem~2.29]{simon-NIP} if $T$ is NIP then each one of these conclusions is equivalent to generic stability for an $A$-invariant type.
\end{fact}

\begin{proposition}\label{prop:generically stable compressible types are l-isolated}
  Let $p \in S(\U)$ be generically stable over $A \subseteq  \U$, and suppose $p|_A$ is compressible.
  Then $p|_A$ is l-isolated.
\end{proposition}

For the proof we will need the following observation. Recall the notations from \cref{sec:compressible types}.

\begin{remark} \label{rem:being l-isolated is elementary}
  Suppose $\tp(a/A)$ is l-isolated, and $(A,a) \equiv (A',a')$. Then $\tp(a'/A')$ is l-isolated. Indeed, if $\zeta(x,d) \in p$ isolates $p \restriction \phi$, then
  \[(A,a) \vDash \exists d \in A^z (\zeta(a,d) \land \forall y 
  (\bigwedge_{\epsilon<2}(\phi(a,y)^\epsilon \rightarrow \forall x (\zeta(x,d) 
  \rightarrow \phi(x,y)^\epsilon)))).\]
  Thus, the same is true in $(A',a')$, which suffices.
\end{remark}

\begin{proof}[Proof of \cref{prop:generically stable compressible types are l-isolated}]
  Suppose $p \in S(\U)$ is generically stable over $A$ and $p|_A$ is compressible. Let $a \vDash p|_A$, and let $M$ be a model containing $Aa$. Let $(M',A',a) \succ (M,A,a)$ be an $|M|^+$-saturated extension (in a language with a predicate $P$ for $A$ and constant symbols $a$), and let $(M'',A'',a) \succ (M',A',a)$ be an $|M'|^{+}$-saturated extension (with $M'' \subseteq \U$). Since $p$ is definable over $A$ by \cref{fac:generically stable -> dfs}(1), it follows that $p|_{A''} = \tp(a/A'')$.

  Let $\phi(x,y)$ be any formula. Note that $(A'',a) \succ (A',a) \succ (A,a)$ and that $(A'',a)$ is $|A'|^+$-saturated, so by \cref{{fac:compressiblity preserved under elem eq},{fac:compressible in terms of elementary extension}}, there is some $d \in (A'')^z$ and some formula $\zeta(x,z)$ such that $\zeta(a,d)$ holds (so $\zeta(x,d)\in p$) and $\zeta(x,d)\vdash (p\restriction \phi)|_{A'}$. By \cref{fac:generically stable -> dfs}(3) and compactness there is some $N<\omega$ such that for every Morley sequence $\sequence{a_i}{i<N}$ of $p$ over $A$, $\zeta(a_i,d)$ holds for some $i<N$, and hence $a_i \vDash (p\restriction \phi)|_{A'}$.

  Let $\sequence{a_i}{i<n}$ be a Morley sequence of $p$ over $A$ of maximal 
  length such that $a_i \not \vDash (p\restriction \phi)|_{A'}$ for all $i<n$. 
  For $i<n$, let $c_i \in (A')^y$ and $\epsilon_i < 2$ be such that 
  $\phi(a_i,c_i)^{\epsilon_i}$ holds but $\neg\phi(x,c_i)^{\epsilon_i} \in p$. 
  Then the following set of formulas over $M'$ is inconsistent:
  \[p^{(n+1)}(x_0,\dots,x_n)|_A \cup \set{\phi(x_i,c_i)^{\epsilon_i}}{i<n} \cup \neg \theta(x_n)\]
  where $\theta(x) = \forall y \in P(\phi(a,y)\leftrightarrow \phi(x,y))$. By compactness (and saturation of $M'$), there is some formula $\psi(x_0,\dots,x_n) \in p^{(n+1)}|_A$ such that $\psi \cup \set{\phi(x_i,c_i)^{\epsilon_i}}{i<n} \vdash \theta(x_n)$.


  Let $\chi(x) = \exists x_0\dots x_{n-1} (\bigwedge_{i<n}\phi(x_i,c_i)^{\epsilon_i} \land \psi(x_0,\dots,x_{n-1},x))$. Since $M'$ is a model, $\chi(x) \vdash (p\restriction \phi)|_{A'}$. Also, since $\psi(a_0,\dots,a_{n-1},x) \in p$, it follows that $\chi(x) \in p|_{A'}$.

  Since $\phi$ was arbitrary, this means that $p|_{A'}$ is l-isolated, and hence by \cref{rem:being l-isolated is elementary}, we are done.
\end{proof}

It is convenient to use the following definition.
\begin{definition} \label{def:GS types over models}
  Suppose $M \vDash T$. A type $p \in S(M)$ is \defn{generically stable} if it has a global $M$-invariant extension which is generically stable over $M$.
\end{definition}

\begin{remark}
  If $p\in S(M)$ is generically stable then it has a unique $M$-invariant extension by \cite[Proposition 2.1(iii)]{AnandPredrag}.
\end{remark}

Note that if $p\in S(M)$ then it has a global $M$-invariant extension (e.g., a coheir). Thus, together with \cite[Proposition~3.4]{MR4033642}, we get the following fact.

\begin{fact} \label{fac:stable => GS over model}
  If $M$ is a model and $p \in S(M)$ is stable then $p$ is generically stable.
\end{fact}

It follows that when the base is a model, \cref{lem:comp stable type}(i.b) is implied by \cref{prop:generically stable compressible types are l-isolated}.

\section{Rounded averages of compressible types and applications}

Let $\Maj$ be the majority rule Boolean operator, i.e., for truth values $P_0,\dots,P_{n-1}$, let
\[\Maj_{i < n} P_i =
\bigvee_{\substack{I_0 \subseteq n \\ |I_0| > n/2}}  \bigwedge_{i \in I_0} P_i.\]
We just write $\Maj_i$ if $n$ is clear.

More generally, for $\alpha \in (0,1)$, let $\Maj^\alpha$ be the ``greater than an $\alpha$-fraction'' Boolean operator, i.e.,
\[\Maj^\alpha_{i < n} P_i =
\bigvee_{\substack{I_0 \subseteq n \\ |I_0| > \alpha n}}  \bigwedge_{i \in I_0} P_i.\]

\begin{definition}\label{def:rounded average}
  Suppose $\phi(x,y)$ is a formula, $B \subseteq \U^y$ and $p_0(x),\ldots ,p_{n-1}(x) \in S_\phi(B)$.

  The \defn{rounded average} of  $p_0(x),\ldots ,p_{n-1}(x) \in S_\phi(B)$ is the following (possibly inconsistent) collection of formulas 
  \[\ravg{\varsequence{p_i}{i<n}}:=\set{\phi(x,b)^\eps}{b \in B, \eps<2, \Maj_{i<n} (\phi(x,b)^\eps \in p_i(x))}.\]

  More generally, for $\alpha \in [\frac{1}{2},1)$, the \defn{$\alpha$-rounded average} is the set
  \[\ravga{\varsequence{p_i}{i<n}}{\alpha}:=\set{\phi(x,b)^\eps}{b \in B, \eps<2, \Maj^\alpha_{i<n} (\phi(x,b)^\eps \in p_i(x))}.\]
\end{definition}

The main result of this section is:
\begin{theorem} \label{thm:rounded average of compressible intro}
  Let $\phi(x,y)$ be an NIP formula and suppose $\alpha \in [1/2,1)$.
  Then there exist $n$ and $k$ depending only on $\vc(\phi)$ and $\alpha$ such that for $A \subseteq  \U^y$,
  any $p \in S_{\phiopp}(A)$ is the $\alpha$-rounded average of $n$ types in 
  $S_{\phiopp\down k}(A)$.
\end{theorem}
(We give a more precise and general statement in \cref{thm:rounded average of compressible}, allowing a partial type $\pi(x)$.)

We give some applications:
\begin{itemize}
  \item Uniformity of honest definitions for NIP formulas, see \cref{def:honest definitions}. This is \cref{cor:luhds}.
  \item Uniform definability of pseudofinite types, see \cref{thm:pseudofinite types are definable}.
\end{itemize}

\subsection{Superdensity}
In this section we isolate a sufficient condition for proving \cref{thm:rounded average of compressible intro}, which uses the $(p,q)$-theorem (see \cref{fac:(p q) theorem}). We then apply it to retrieve UDTFS in \cref{cor:UDTFS}, as a prelude to the proof of the uniformity of honest definitions in \cref{cor:luhds}.

\begin{definition} \label{def:p,q}
  Suppose $q \leq p<\omega$. A set system $(X,\mathcal{F})$ has the \defn{$(p,q)$-property} if for any $S \subseteq \mathcal{F}$ such that $|S| \geq p$, there exists $S_0 \subseteq S$ of size $|S_0| \geq q$ such that $\bigcap S_0 \neq \emptyset$.
\end{definition}

\begin{fact}\label{fac:(p q) theorem} \cite{MR2060639}
  (The $(p,q)$-theorem) There exists a function $\Npq : \N^2 \rightarrow  \N$ such that for any $q\leq p<\omega$, if $(X,\mathcal{F})$ is a finite set system with the $(p,q)$-property such that every $s\in \mathcal{F}$ is nonempty and $\vc^*(\mathcal{F})<q$, then there is $X_0 \subseteq X$ of size $|X_0| = \Npq(p,q)$ such that $X_0 \cap s \neq \emptyset$ for all $s\in \mathcal{F}$.
\end{fact}

We isolate from the proof of \cite[Corollary~6.11]{simon-NIP} the following immediate generalisation of the $(p,q)$-theorem to infinite set systems.

\begin{lemma} \label{lem:compactness pq}
  Let $\phi(x,y)$ be NIP.
  Let $p \geq  q > \vc^*(\phi)$ be integers, and let $N = \Npq(p,q)$.
  Let $A \subseteq  \U^x$ and $B \subseteq  \U^y$.
  Suppose that $\phi(A,b) \neq  \emptyset $ for every $b \in B$,
  and that for every $B_0 \subseteq  B$ with $|B_0| = p$ there exists $B_1 \subseteq  B_0$ with $|B_1| = q$ such that for some $a \in A$ we have $\vDash  \bigwedge_{b \in B_1} \phi(a,b)$.

  Then $\set{\bigvee_{i<N} \phi(x_i,b)}{b \in B}$
  is finitely satisfiable in $A$.
\end{lemma}
\begin{proof}
  By the definition of finite satisfiability, it suffices to see this in the case that $B$ is finite; but this case is a direct consequence of the $(p,q)$-theorem (\cref{fac:(p q) theorem}).
\end{proof}


Suppose $\phi(x,y)$ is a formula, $A \subseteq \U^x$ and $N \in \N$. For variables $\bar{x}=\varsequence{x_i}{i<N}$ of the same sort of $x$, we denote by $\Sdfsl{\bar{x}}{A}{\U^y}{\phi}$ the space of $\Delta$-types in $\bar{x}$ over $\U^y$ which are finitely satisfiable in $A$, where $\Delta = \set{\phi(x_i,y)}{i<N}$. If $p(y) \in S_\phiopp(A)$ then for any $q \in \Sdfsl{\bar{x}}{A}{\U^y}{\phi}$, the product $q(\bar{x}) \otimes p(y)$ is the partial type $q(\bar{x}) \cup p(y) \cup \set{\phi(x_i,y)^{\epsilon_i}}{i<N}$ where $\epsilon_i<2$ is the truth value of $\phi(a_i,b)$ for some (any) $b \vDash p$ and $\varsequence{a_i}{i<N} \vDash q|_{b}$. Note that this is well-defined.

For $b \in \U^y$, $N \in \N$ and $S \subseteq S_\phiopp(A)$, we consider the following condition:

\noindent\hypertarget{superdense}{$(\dagger)_{b,N,S}$} $\left\{
\begin{minipage}{0.8\textwidth}
  \vspace{0.5em}
  For every
  $q(\bar{x}) \in \Sdfsl{\bar{x}}{A}{\U^y}{\phi}$ where $\bar{x}=\varsequence{x_i}{i<N}$,
  there is $p \in S$
  such that
  \[q(\bar{x}) \otimes p(y) \vdash  \bigwedge_{i<N} (\phi(x_i,y) \leftrightarrow \phi(x_i,b)).\]
  \vspace{0.5em}
\end{minipage}
\right.$

\begin{definition} \label{def:superdensity}
  Suppose $\phi(x,y)$ is a formula and $A \subseteq \U^x$. A set $S \subseteq S_\phiopp(A)$ is \defn{superdense} in $S_\phiopp(A)$ if \hyperlink{superdense}{$(\dagger)_{b,N,S}$} holds for every $b \in \U^y$ and $N \in \N$.
\end{definition}

\begin{remark}
  By considering realised types in $A$, it follows that any superdense set $S \subseteq S_\phiopp(A)$ is also dense.
\end{remark}



\begin{lemma} \label{lem:superdensity sufficient condition for rounded average}
  Let $\alpha \in [1/2,1)$ and let $d \in \N$. Let $n \in \N$ be such that $(1-\alpha)n>2^{d+1}-1$, and let $N = \Npq(n,2^{d+1})$. Fix some formula $\phi(x,y)$ such that $\vc(\phi)\leq d$.
  Suppose $A \subseteq  \U^x$, $b \in \U^y$ and $S \subseteq  S_{\phiopp}(A)$ satisfy  \hyperlink{superdense}{$(\dagger)_{b,N,S}$}.


  Then $r:=\tp_\phiopp(b/A)$ is the $\alpha$-rounded average of $n$ elements of $S$.
\end{lemma}

\begin{proof}
  Since $\alpha \geq 1/2$, $r = \ravga{\varsequence{p_i}{i<n}}{\alpha}$ iff for all $a \in A$, $|\set{i<n}{\phi(a,y)\in r \Leftrightarrow \phi(a,y) \in p_i}| > \alpha n$. 
  Let $\phi'(x,y) = \phi(x,y) \leftrightarrow  \phi(x,b)$ (as a formula over $\U$). In this notation, we must prove that for some $p_0, \ldots, p_{n-1} \in S$, for all $a\in A$, $|\set{i<n}{p_i(y) \vdash \phi'(a,y)}| > \alpha n$.


  Note that $\vc(\phi') = \vc(\phi) \leq d$; indeed, $\phi'$ and $\phi$ shatter the same subsets of $\U^x$.

  Now assume the conclusion fails. In particular, $r \notin S$ (else it is the rounded average of $n$ copies of itself).
  Let $B \subseteq  \U^y$ be a set of realisations of the types in $S$. By assumption, for every $B_0 \subseteq  B$ with $|B_0| = n$, there exists $a\in A$ and $B_1 \subseteq  B_0$ with $|B_1| \geq (1-\alpha)n > 2^{d+1}-1 \geq \vc^*(\phi')$ such that
  $\neg \phi'(a,c)$ holds for all $c \in B_1$ (the last inequality follows from \cref{rem:bound on vc-dimension of dual}).

  Note that for $b' \in B$ we have $\tp_\phiopp(b'/A) \neq  \tp_\phiopp(b/A)$ since $\tp_\phiopp(b/A) \notin S$, so $\neg\phi'(A,b') \neq  \emptyset$.

  So \cref{lem:compactness pq} applies to $\neg \phi'$ (with $p=n$, $q=2^{d+1}$; note that $\vc^*(\phi')=\vc^*(\neg \phi')$) and hence $\pi(\bar{x}):=\set{\bigvee_{i<N} \neg \phi'(x_i,c)}{c \in B}$ is finitely satisfiable in $A$, where  $\bar{x} = \varsequence{x_i}{i<N}$.

  Extend $\pi$ to
  $q(\bar{x}) \in\Sdfsl{\bar{x}}{A}{\U^y}{\phi}$ (formally, first extend $\pi$ to a global type finitely satisfiable in $A$, and then restrict to a global $\set{\phi(x_i,y)}{i<N}$-type $q$. Then note that $q$ implies $\pi$).
  Then for all $p \in S$, we have
  $q (\bar{x}) \otimes p(y) \vdash  \bigvee_{i<N} \neg \phi'(x_i,y)$. However, \hyperlink{superdense}{$(\dagger)_{b,N,S}$} implies that for some $p \in S$, $q(\bar{x}) \otimes p(y) \vdash  \bigwedge_{i<N} \phi'(x_i,y)$, contradiction.
\end{proof}

\begin{remark}
  When $\alpha = 1/2$, $n: = 2^{\vc(\phi)+2} - 1$ and $N: = \Npq(n,2^{\vc(\phi)+1})$ work in \cref{lem:superdensity sufficient condition for rounded average}.
\end{remark}

\begin{remark}\label{rem:ravg of distinct}
  From the proof of \cref{lem:superdensity sufficient condition for rounded average}, we get something slightly stronger (under the same assumptions): either $r \in S$, or $r$ is an $\alpha$-rounded average of \emph{distinct} types in $S$.
\end{remark}

\begin{remark}
  \cref{lem:superdensity sufficient condition for rounded average} admits a partial converse, for an arbitrary formula $\phi(x,y)$ and any $n$ and any $N$:
  letting $\bar{x} = \varsequence{x_i}{i<N}$, if $\tp_\phi(b/A)$ is the rounded average of $n$ elements from $S \subseteq  S^y_\phiopp(A)$,
  then for every
  $q (\bar{x}) \in  \Sdfsl{\bar{x}}{A}{\U^y}{\phi}$
  there is $p(y) \in S$
  such that
  \[q (\bar{x}) \otimes p(y) \vdash  \Maj_{i<N} (\phi(x_i,y) \leftrightarrow  \phi(x_i,b)).\]

  Indeed, suppose $\tp_\phiopp(b/A)=\ravg{\varsequence{\tp(c_j/A)}{i<n}}$ where $\tp(c_j/A)\in S$ for $j<n$
  and
  $q (\bar{x}) \in  \Sdfsl{\bar{x}}{A}{\U^y}{\phi}$.
  If the conclusion fails, we get that for each $j<n$
  \[q \vdash \neg \Maj_{i<N}  (\phi(x_i,c_j) \leftrightarrow  \phi(x_i,b)).\]

  By finite satisfiability, there is $(a_0,\ldots ,a_{N-1}) \in A^N$ satisfying this for every $j<n$.
  Hence $|\set{(i,j)\in N \times n}{\phi(a_i,c_j) \leftrightarrow \phi(a_i,b)}|\leq \frac{1}{2} n N$.
  Then by the pigeonhole principle,
  for some $i<N$ we have
  \[\vDash  \neg \Maj_{j<n} (\phi(a_i,c_j) \leftrightarrow \phi(a_i,b)),\]
  contradicting $\tp_\phiopp(b/A)$ being the rounded average of the $\tp(c_i/A)$.
\end{remark}

In \cref{sec:rounded average proof} we will prove that $S_{\phiopp\down \star}(A)$ is superdense in $S_\phiopp(A)$ and even in a uniform way, as in the proof of \cref{cor:localComp}(iii), which will imply \cref{thm:rounded average of compressible intro} by \cref{lem:superdensity sufficient condition for rounded average}. In the finite case we can already conclude the following, basically because superdensity is the same as density when $A$ is finite.

\begin{corollary} \label{cor:rounded average of compressible finite domain}
  Fix $\alpha \in [1/2,1)$ and some $d \in \N$. Then there are $k,n \in \N$ (depending only on $d,\alpha$) such that if $\phi(x,y)$ is a formula with $\vc(\phi)\leq d$, then for any finite $A \subseteq \U^x$, every $r(y)\in S_{\phiopp}(A)$ is the $\alpha$-rounded average of $n$ types in $S_{\phiopp\down k}(A)$.
\end{corollary}

\begin{proof}
  Let $n,N$ be as in \cref{lem:superdensity sufficient condition for rounded average}, and let $k=\kc(d)+N$. Fix some $b \in \U^y$. By \cref{lem:superdensity sufficient condition for rounded average}, to show the conclusion for $\tp_{\phiopp}(b/A)$, it is enough to show \hyperlink{superdense}{$(\dagger)_{b,N,S}$} with $S =S_{\phiopp\down k}(A)$. But every $q(\bar{x}) \in \Sdfsl{\bar{x}}{A}{\U^y}{\phi}$ is realised in $A$ (since $A$ is finite), 
  so it boils down to showing that for every $\bar{a} \in A^N$, there is some $p(y) \in S_{\phiopp\down k}(A)$ such that $p(y) \vdash \bigwedge_{i<N }(\phi(a_i,y) \leftrightarrow \phi(a_i,b))$. This follows from the choice of $k$ and \cref{cor:localComp}(ii) applied to $\phiopp$.
\end{proof}
As a corollary we retrieve UDTFS. First recall the definition.

\begin{definition}[UDTFS] \label{def:(UDTFS)}
  We say that
  $\phi(x,y)$ has \defn{uniform definability of types over finite
  sets} (\defn{UDTFS}) if there exists a formula $\psi(x,z)$
  such that for every finite set $A\subseteq \U^{x}$
  with $|A| \geq 2$ the following holds: for every $b \in \U^y$
  there exist $c\in A^{z}$ such that $\psi(A,c) = \phi(A,b)$.
\end{definition}

Every formula with UDTFS is easily NIP (see e.g., the proof of Theorem~14 in \cite{UDTFS}). The proof of UDTFS for NIP formulas in \cite{UDTFS} roughly goes by showing \cref{cor:rounded average of compressible finite domain} with $\alpha = 1/2$, and deducing UDTFS from that (this is not stated explicitly in this language, see the proof of Theorem~14, (1) implies (2), (3) there). We omit the details here since we will prove uniformity of honest definitions in \cref{cor:luhds} below. In \cref{thm:pseudofinite types are definable} we will extend UDTFS to pseudofinite types (see \cref{def:pseudofinite type}).

\begin{corollary} \label{cor:UDTFS}
  The formula $\phi(x,y)$ is NIP iff it has UDTFS.
\end{corollary}

We do point out that the proof here is, at least conceptually, simpler than the proof in \cite{UDTFS}: both proofs use the finite version of \cref{cor:density of comp in countable NIP}, but here the only other ingredient is the $(p,q)$-theorem, while there both the VC-theorem and von Neumann's minimax theorem are used.

\subsection{A variant of Ramsey's theorem for finite subsets}
Here we will prove a variant of Ramsey's theorem for finite subsets of a cardinal. This result generalises \cref{rem:cofinal finite colouring} for $(\pfin(\kappa), \subseteq)$ in the same way that Ramsey's theorem generalises the pigeonhole principle. It will be used in the proof of \cref{thm:rounded average of compressible intro}.

For a partial order $(X,\leq)$ and $n \in \N$, let $X^{n}_{<} = \set{(x_0,\ldots ,x_{n-1}) \in X^n}{x_0 <  \ldots  < x_{n-1}}$ be the set of ordered chains of size $n$ (in short, \defn{$n$-chains}).

For $\kappa$ a cardinal, let $\pfin(\kappa)$ be the set of finite subsets of $\kappa$, partially ordered by inclusion.

Say $f : \pfin(\kappa) \rightarrow  \pfin(\kappa)$ is
\defn{strictly increasing} if
$s \subsetneq  t \Rightarrow  f(s) \subsetneq  f(t)$ for all $s,t \in \pfin(\kappa)$, and say $f$ is
\defn{cofinal} if for all $s \in \pfin(\kappa)$ there is $t \in \pfin(\kappa)$ such that $f(t) \supseteq s$. Note that the image of an $n$-chain under a strictly increasing map is an $n$-chain.

\begin{proposition} \label{prop:Ramsey for Pfin}
  Let $\kappa$ be an infinite cardinal, $0<n \in \N$
  and let $c : \pfin(\kappa)^{n}_{<} \rightarrow  r<\omega$ be a finite colouring of the $n$-chains.
  Then there is a
  strictly increasing
  cofinal map $f : \pfin(\kappa) \rightarrow  \pfin(\kappa)$ such that the image of all $n$-chains $f(\pfin(\kappa)^{n}_{<})$ is monochromatic, i.e., $|(c\circ f)(\pfin(\kappa)^{n}_{<})| = 1$.
\end{proposition}
\begin{proof}
  \newcommand{\emptychain}{\left<\right>}

  Denote by $M$ the structure $(\pfin(\kappa),\subseteq, \sequence{P_k}{k<r})$ where $P_k = c^{-1}(k) \subseteq M^{n}_{<}$ for $k<r$. Let $N \succ M$ be an $|M|^+$-saturated extension. Let $\pi=\set{x\supsetneq s}{s \in M}$. As $\pi$ is finitely satisfiable in $M$, there is $q \in \Sfs{M}{N}$ extending $\pi$. Let $(a_{n-1},\dots,a_{0}) \vDash q^{(n)}|_{M}$ and let $\bar{a} = (a_0, \dots, a_{n-1})$. Note that $\bar{a} \in N^n_<$ (because for all $a \vDash q|_M$ and all $b \in M$, $b \subsetneq a$), and hence for some $k<r$, $\bar{a} \in P^N_k$. We claim that this colour $k$ works.



  For $m<\omega$, let $S_m = \set{s \in \pfin(\kappa)}{|s|=m}$. By recursion on $m<\omega$ we define $f\restriction S_m$ such that for any $s \in S_m$:

  \begin{enumerate}[(i)]
    \item for any $1\leq i\leq n$ and $i$-chain $s_0 \subsetneq  \ldots  \subsetneq  s_{i-1} = s$, $(f(s_0),\dots,f(s_{i-1}),a_i,\dots,a_{n-1}) \in P_k^N$.
    \item $f(s) \supseteq s$ and $f(s) \supsetneq f(t)$ for all $t \subsetneq s$.
  \end{enumerate}

  The construction is possible because $q$ is finitely satisfiable in $M$ and since there only finitely many conditions to fulfill for each $s \in \pfin(\kappa)$ (and because of the choice of $\bar{a}$):
  given $s \in S_m$ and a chain as in (i), by induction we have $q(x) \vdash P_k(f({s_0}),\dots,f(s_{i-2}),x,a_i,\dots,a_{n-1})$; since there are only finitely many such chains to consider and also $q(x) \vdash x \supsetneq f(t)\cup s$ for any $t \subsetneq s$, we can find $f(s)$ by finite satisfiability.


  Now the $i=n$ case of (i) implies that the image under $f$ of any $n$-chain has colour $k$. Meanwhile (ii) implies that $f$ is cofinal, and that $f$ is strictly increasing.
\end{proof}

\subsection{The proofs of superdensity and of \texorpdfstring{\hyperref[thm:rounded average of compressible intro]{\cref{thm:rounded average of compressible intro}}}{\autoref{thm:rounded average of compressible intro}}}\label{sec:rounded average proof}

In this section we will prove \cref{thm:rounded average of compressible intro}, by proving superdensity of compressible types in a uniform way.

\begin{proposition}[Superdensity of $\star$-compressible types] \label{prop:superdensity}
  Define $\ksd : \N^2 \rightarrow  \N$ by $\ksd(1,d) := \kc(d) + 2d + 2$,
  and $\ksd(n,d) := n\cdot \ksd(1,2^{\Bvc(n,2^{d+1})+1})$ for $n\neq 1$ (see \cref{def:binom vc}).

  Let $\phi(x,y)$ be a formula, $d \in \N$, and assume that $\vc(\phi)\leq d$.
  Let $A \subseteq  \U^x$.
  Let $\pi(y)$ be a (small) partial type.
  Let $b \in  \pi(\U)$.
  For $0<n \in \N$, let $S = S_{\phiopp\down \ksd(n,d)}^\pi(A)$.

  Then \hyperlink{superdense}{$(\dagger)_{b,n,S}$} holds.

\end{proposition}
\begin{proof}
  Let $\bar{x}=\varsequence{x_i}{i<n}$ and let $q(\bar{x}) \in \Sdfsl{\bar{x}}{A}{\U^y}{\phi}$.

  We first reduce to the case
  \begin{equation}\tag{*} \label{eqn:n=1 case}
    n=1 \text{ and } q \vdash  \phi(x,b).
  \end{equation}

  For $i<n$, let $\epsilon_i<2$ be such that $q \vdash \phi(x_i,b)^{\epsilon_i}$.
  Let $\phi'(\xx,y) = \bigwedge_{i<n} \phi(x_i,y)^{\eps_i}$,
  and consider the global $\phi'$-type $q'(\xx) \in S^{\xx}_{\phi'}(\U^y)$ implied by $q$,
  which is finitely satisfiable in $A$.

  Now $\vc^*(\neg\phi) = \vc^*(\phi)$, and so by \cref{lem:vcBool},
  and \cref{rem:bound on vc-dimension of dual}, we have
  $\vc(\phi') < 2^{\vc^*(\phi')+1} \leq  2^{\Bvc(n,\vc^*(\phi))+1} <
  2^{\Bvc(n,2^{\vc(\phi)+1})+1} \leq 2^{\Bvc(n,2^{d+1})+1} =: d'$. (Note that $\Bvc$ is increasing in the second variable.)
  Note that $q' \vdash \phi'(\xx,b)$.

  Let $k=\ksd(1,d')$.
  Assuming the proposition in the \cref{eqn:n=1 case} case,
  we obtain
  $p' \in S_{(\phi')^{\opp}\down k}^\pi(A^n)$ such that
  $q'(\xx) \otimes p'(y) \vdash  \phi'(\xx,y)$.

  \begin{claim*}
    $p'$ implies a complete type $p \in S_{\phiopp\down \ksd(n,d)}^\pi(A)$ (which implies $p'$).
  \end{claim*}
  \begin{proof}
    We first show that if $a \in A$ then $p'(y) \vdash  \phi(a,y)^\eps$ for some $\eps$.

    Since $q'(\xx) \otimes p'(y) \vdash  \phi'(\xx,y)$ and since $q'$ is finitely satisfiable in $A$, there is some $(a_0,\ldots ,a_{n-1}) \in A^n$ with $p'(y) \vdash  \phi'(a_0,\ldots ,a_{n-1},y)$ (take some $c \vDash p'$, then since $\phi'(\xx,c)\in q'(\xx)$, there is some $\bar{a}:=(a_0,\dots,a_{n-1}) \in A^n$ such that $\phi'(\bar{a},c)$ holds, but as $p'$ is complete, the same is true for any such $c$).
    Now, if $p'(y) \vdash \phi'(a,a_1,\dots,a_{n-1},y)$ then $p'(y) \vdash \phi(a,y)^{\epsilon_0}$. If $p'(y) \vdash \neg \phi'(a,a_1,\dots,a_{n-1},y)$ then by the choice of $(a_0,\dots,a_{n-1})$, necessarily $p'(y)\vdash \phi(a,y)^{1-\epsilon_0}$.

    So $p'$ implies a complete type $p \in S_{\phiopp}^\pi(A)$. Let $A_0 \subseteq A$ be a finite subset. Then for some finite $A_0' \subseteq A^n$, $p'|_{A_0'} \vdash p|_{A_0}$. By $k$-compressibility of $p'$ modulo $\pi$, there is some $A_1' \subseteq A^n$ of size $k$ such that $p'|_{A_1'}\vdash_\pi p'|_{A_0'}$. Let $A_1 = \set{a\in A}{a \text{ appears in some } \bar{a}\in A_1'}$. Then $|A_1| \leq nk$ and $p|_{A_1}\vdash p'|_{A_1'}\vdash_{\pi} p|_{A_0}$, so $p$ is $nk$-compressible modulo $\pi$. Since $nk = \ksd(n,d)$, we are done.
  \end{proof}
  Now, $q(\xx) \otimes p(y) \vdash  q'(\xx) \otimes p'(y)$ and $q'(\xx) \otimes p'(y) \vdash  \phi'(\xx,y) =  \bigwedge_i \phi(x_i,y)^{\eps_i}$ so $p$ is as required in \hyperlink{superdense}{$(\dagger)_{b,n,S}$}.

  It remains to prove the proposition assuming \cref{eqn:n=1 case}, so assume that $n=1$ and $q(x) \vdash  \phi(x,b)$.

  If $A$ is finite then $q$ is realised in $A$, and we conclude (as in \cref{cor:rounded average of compressible finite domain}) by
  \cref{cor:localComp}(ii) applied to $\phiopp$ (note that $\kc(d) + 2d + 2 \geq \kc(d)+1$ which would be enough in this case).
  So suppose $\kappa := |A| \geq  \aleph_0$.

  Let $\Ffin(\kappa)$ be the filter on $\pfin(\kappa)$ generated by $\set{X_s}{s \in \pfin(\kappa)}$ where $X_s = \set{t\in \pfin(\kappa)}{t \supseteq s}$.

  By \cite[Lemma~2.9]{simon-invariant},
  $q$ is the limit of a sequence of types realised in $A$:
  \[q = \lim_{s\rightarrow \Ffin(\kappa)}(\tp_\phi(a_s/\U^y))\]
  where $a_s \in A$. (Note that in \cite[Lemma~2.9]{simon-invariant}, the type is over a model, but the same statement, with the same proof, works also over a set.)

  This means that for any $c \in \U^y$, $\phi(x,c) \in q$ iff $\set{s}{\phi(a_s,c)} \in \Ffin(\kappa)$ iff for some $s \in \pfin(\kappa)$, $\phi(a_t,c)$ holds for all $t \supseteq s$. (Since $q$ is complete, in fact the left-to-right implication suffices: $q = \lim_{s\rightarrow \Ffin(\kappa)}(\tp_\phi(a_s/\U^y))$ iff we have that $\set{s}{\phi(a_s,c)} \in \Ffin(\kappa)$ whenever $\phi(x,c) \in q$.)

  Since $q \vdash  \phi(x,b)$, we may assume that $\vDash  \phi(a_s,b)$ for all $s$ (indeed, if $s_0$ is such that $\phi(a_s,b)$ for all $s \supseteq s_0$, then we can ensure this by replacing $a_s$ with $a_{s\cup s_0}$).

  Let $m= 2\vc(\phi)+3$.
  Let $c^\pi_m$ be the $2^{2^m}$-colouring of $m$-chains in $\pfin(\kappa)$ indicating which Boolean combinations of the corresponding $m$ instances of $\phi$ are consistent with $\pi$, i.e., $c^\pi_m(s_0,\ldots ,s_{m-1}) = \set{\sequence{\epsilon_i}{i<m} \in 2^m}{\bigwedge_{i<m} \phi(a_{s_i},y)^{\eps_i} \not\vdash _\pi \bot}$.
  By \cref{prop:Ramsey for Pfin}, we may assume that $c^\pi_m$ is constant; indeed, if $f$ is as in \cref{prop:Ramsey for Pfin}, we may replace $a_s$ with $a_{f(s)}$, and then $q$ will still be the limit since $f$ is cofinal. We will refer to the property that $c^\pi_m$ is constant as  \defn{$c^\pi_m$-homogeneity}.

  Identify $i \in \N$ with $\{0,\ldots ,i-1\} \in \pfin(\kappa)$.
  Take a $\phiopp$-type $p_0(y) \in S_\phiopp^\pi((a_i)_{i<m})$ which first strictly alternates maximally and then is constantly true;
  i.e., $p_0(y) \vdash  \phi(a_i,y) \leftrightarrow  \neg \phi(a_{i+1},y)$ for $i < l$
  and $p_0(y) \vdash  \phi(a_i,y)$ for $i \in [l,m)$,
  and $l<m$ is maximal such that such a type exists.
  Note that $\tp(b/(a_i)_{i<m})$ is of this form with $l=0$, so some such $p_0$ exists (here we use the fact that $b \vDash \pi$).
  By $c^\pi_m$-homogeneity (in fact $c^\pi_{\vc(\phi)+1}$-homogeneity is enough) and the usual argument for bounding alternation number (see \cite[Lemma~2.7]{simon-NIP}), we have $l \leq  2\vc(\phi) < m-2$. 

  Define
  \[p_1(y) = p_0(y)|_{\{a_0,\ldots ,a_l\}} \cup \set{\phi(a_s,y)}{l \subseteq  s \in \pfin(\kappa)},\]
  and let $A'= \{a_0,\ldots ,a_l\} \cup \set{a_s}{l \subseteq  s \in \pfin(\kappa)}$ be the domain of $p_1$.

  \begin{claim*}
    $p_1 \in S_{\phiopp\down m-1}^\pi(A')$ and in particular is consistent with $\pi$.
  \end{claim*}
  \begin{proof}
    Suppose $A'_0 \subseteq_{\fin} A'$. Let $s_1 \in \pfin(\kappa)$ strictly contain all sets of the form $l\cup s \in \pfin(\kappa)$ such that $a_s \in A_0'$. Let $s_1 \subsetneq  \ldots  \subsetneq  s_{m-l-2}$ be an $m-l-2$-chain starting with $s_1$. Let
    \[p_0'(y) = p_1(y)|_{\set{a_i}{i \in [0,l] \cup \{s_1,\ldots ,s_{m-l-2}\}}}.\]

    Then by $c^\pi_m$-homogeneity, $p_0'$ is consistent with $\pi$ since $p_0$ is. 

    Suppose $l\subsetneq s_0 \subsetneq s_1$. We claim that $p_0' \vdash _\pi \phi(a_{s_0},y)$.


    Otherwise, by $c^\pi_m$-homogeneity,
    \[p_1(y)|_{\set{a_i }{i \in [0,l] \cup [l+2,m)}}
    \cup \{ \neg\phi(a_{l+1},y) \}\]
    is consistent with $\pi$. Since $l+2 < m$, this contradicts the maximality of $l$.

    In particular, $p_0' \vdash_\pi p_1(y)|_{A_0'}$ and the latter is consistent with $\pi$.
  \end{proof}

  Now by \cref{cor:localComp}(ii) applied to $\phiopp$,
  $p_1$ extends to $p \in S_{\phi\down(m-1) + \kc(d)}^\pi(A)$. Since $m-1+ \kc(d) \leq \ksd(1,d)$, to conclude it is enough to show that $q(x) \otimes p(y) \vdash  \phi(x,y)$.
  But this holds
  since $q = \lim_{s\rightarrow \Ffin(\kappa)}(\tp_\phi(a_s/\U^y))$ and $p(y) \vdash  \phi(a_s,y)$ for all $s \supseteq l$.
\end{proof}

We can now deduce \cref{thm:rounded average of compressible intro}.

\begin{theorem} \label{thm:rounded average of compressible}
  Let $d \in \N$ and $\alpha \in (1/2,1]$.
  Then there exist $n$ and $k$ depending only on $d, \alpha$ such that the following holds.
  If $\phi(x,y)$ is a formula such that $\vc(\phi)\leq d$, then
  for any $A \subseteq  \U^x$ and a (small) partial type $\pi(y)$,
  any $p \in S_{\phiopp}^\pi(A)$ is the $\alpha$-rounded average of $n$ types in $S_{\phiopp\down k}^\pi(A)$.

  Namely, we may take $n:= \min\set{m\in \N}{(1-\alpha)m>2^{d+1}-1}$ and $k:=\ksd(\Npq(n,2^{d+1}),d)$.
\end{theorem}
\begin{proof}
  By \cref{lem:superdensity sufficient condition for rounded average} it is enough to show \hyperlink{superdense}{$(\dagger)_{b,N,S}$} where $N= \Npq(n,2^{d+1})$ and $S = S_{\phiopp\down \ksd(N,d)}^\pi(A)$, which follows  by \cref{prop:superdensity}.
\end{proof}

\begin{remark}
  By \cref{rem:ravg of distinct}, in the context of \cref{thm:rounded average of compressible}, if $p$ is not $k$-compressible then it is an $\alpha$-rounded average of $n$ \emph{distinct} types in $S_{\phiopp\down k}^\pi(A)$.
\end{remark}

We give some immediate corollaries.

\begin{corollary}\label{cor:many compressible types}
  If $\phi(x,y)$ is NIP and $A \subseteq \U^y$ then $|S_{\phiopp \down \star}(A)| + \aleph_0 = |S_{\phiopp}(A)|+\aleph_0$.

  Moreover,  $|S_{\phiopp \down k}(A)| + \aleph_0 = |S_{\phiopp}(A)|+ \aleph_0$ for $k$ from \cref{thm:rounded average of compressible} and $S_{\phiopp \down k}(A)$ is finite iff $S_{\phiopp}(A)$ is.
\end{corollary}

We can also improve \cref{lem:stable Comp = Isol Local} to add another equivalence:
\begin{corollary}
  The following are equivalent for an NIP formula $\phi(x,y)$:
  \begin{enumerate}[(i)]
    \item $\phi$ is stable.
    \item For any model $M\vDash T$ and any $k\in \N$, any $p \in S_{\phi\down k} (M^y)$ is isolated.
  \end{enumerate}
\end{corollary}

\begin{proof}
  (i) implies (ii) follows from \cref{lem:stable Comp = Isol Local}.

  $\neg$(i) implies $\neg$(ii): since $\phi$ is not stable, there is some (infinite) model $M$ such that $|S_\phi(M)|>|M|$ (see e.g., \cite[Theorem~8.2.3]{TentZiegler}). By \cref{cor:many compressible types} (applied to $\phiopp$) $|S_{\phi\down k}(M^y)|>|M|$ for some $k\in\N$. We conclude, since there are at most $|M|$ isolated $\phi$-types over $M$.
\end{proof}

\begin{remark} \label{rem:superdensity under stability}
  Suppose $\phi(x,y)$ is stable. Then we can replace $\ksd$ in \cref{prop:superdensity} by a linear (as opposed to exponential, see \cref{rem:Bvc(n k) >= n k}) bound in terms of $n$ with a simpler proof.

  Let $\bar{x}=\varsequence{x_i}{i<n}$, $q(\bar{x}) \in \Sdfsl{\bar{x}}{A}{\U^y}{\phi}$, $\pi(y)$ and $b \in \pi(\U)$ be as there. For $i<n$, let $q_i(x_i) = q\restriction \{\phi(x_i,y)\}$. As $q_i$ is finitely satisfiable in $A$, by \cite[Exercise~8.3.6]{TentZiegler}, $q_i$ is definable by a Boolean combination of instances of $\phiopp$ over $A$ (the exercise assumes that $T$ is stable but this is not necessary). The size of this Boolean combination depends only on $\phi$ (really only on the size of a maximal witness for the order property). Hence there is $l$ depending only on $\phi$, and $a_{i,j} \in A$, $\epsilon_{i,j}<2$ for $j<l$, such that for some formula $\theta_i(y)$ of the form $\bigwedge_{j<l} \phi(a_{i,j},y)^{\epsilon_{i,j}}$, $\theta_i(b)$ holds and if $b'\vDash \theta_i(y)$ then $\phi(x_i,b)\in q$ iff $\phi(x_i,b') \in q$.
  Let $\theta(y) = \bigwedge_{i<n} \theta_i$. 
  Note that $\theta(b)$ holds, so that $\theta$ is consistent with $\pi$.

  By \cref{cor:localComp}(ii) applied to $\phiopp$ (here we could use also the stable counterpart, using the 2-rank), there is some $p(y) \in S_{\phiopp\down ln+\kc(\vc(\phi))}^\pi(A)$ such that $p(y) \vdash \theta(y)$. Thus,
  \[q(\bar{x}) \otimes p(y) \vdash  \bigwedge_{i<n} (\phi(x_i,y) \leftrightarrow \phi(x_i,b)).\]

  Since $q$ was arbitrary, we get \hyperlink{superdense}{$(\dagger)_{b,n,S}$} for $S: = S_{\phiopp\down ln+\kc(\vc(\phi))}^\pi(A)$.
\end{remark}

\subsection{Local uniform honest definitions}

In this section we will prove uniformity of honest definitions for NIP formulas.


\begin{definition}\label{def:honest definitions}
  \cite[Definition~3.16 and Remark~3.14]{simon-NIP} Suppose $\phi(x,y)$ is a formula, $A\subseteq M^{x}$
  is some set and $b \in \U^y$. Say that a formula $\psi(x,z)$ over $\emptyset$ (with $z$ a tuple of variables each of the same sort as $x$) 
   is an \defn{honest definition} of $\tp_\phiopp(b/A)$ if for
  every finite $A_{0}\subseteq A$ there is some $c\in A^{z}$ such
  \[\phi(A_0,b) \subseteq \psi(A,c) \subseteq \phi(A,b).\]
  In other words, for all $a\in A$, if $\psi(a,c)$ holds then so does $\phi(a,b)$
  and for all $a\in A_{0}$ the other direction holds: if $\phi(a,b)$ holds
  then $\psi(a,c)$ holds.
\end{definition}

It is proved in \cite[Theorem 6.16]{simon-NIP}, \cite[Theorem 11]{CS-extDef2}
that if $T$ is NIP then for every $\phi(x,y)$ there is a formula
$\psi(y,z)$ that serves as an honest definition for any type in $S_{\phi}(A)$
provided that $|A|\geq 2$ (by \cite[Remark~16]{CS-extDef2} only some NIP is required of $\phi$ and formulas expressing consistency of Boolean combinations of $\vc(\phi)+1$ instances of $\phi$).
In this section we improve this by proving this result assuming only that $\phi$ is NIP.


\begin{corollary} \label{cor:luhds}
  Let $\phi(x,y)$ be NIP.
  Then there exists $\psi(x,z)$ such that
  if $A \subseteq  \U^x$ with $|A| > 1$
  and $b \in \U^y$,
  then $\psi(x,z)$ is an honest definition of $\tp_\phiopp(b/A)$.

  Namely, \[\psi(x,(\zz,\zz',\zz'')) :=
  \Maj_{i<n} \forall y \left({\bigwedge_{j<k}
  (\phi(z_{i,j},y) \leftrightarrow  (z'_{i,j} = z''_{i,j})) \rightarrow  \phi(x,y)}\right),\]
  where $n$ and $k$ are as in \cref{thm:rounded average of compressible} with $d=\vc(\phi)$ and $\alpha=1/2$.
\end{corollary}
\begin{proof}

  By \cref{thm:rounded average of compressible}, $\tp_{\phiopp}(b/A)$ is the rounded average of $k$-compressible types $p_0,\ldots ,p_{n-1} \in S_{\phiopp}(A)$.

  Now we proceed as in the proof of \cref{lem:adding one more formula}: if $A_0 \subseteq_{\fin}  A$, there are $D_i = (d_{i,j})_{j<k} \subseteq  A$ for $1\leq i\leq n$ such that $p_i|_{D_i} \vdash  p_i|_{A_0}$. Let $c_0,c_1 \in A$ be distinct, let $d'_{i,j} = c_0$, and let $d''_{i,j} = c_0$ if $p_i(y) \vdash  \phi(d_{i,j},y)$ and $d''_{i,j} = c_1$ otherwise, so $p_i|_{D_i}(y)$ is equivalent to $\bigwedge_{j<k} (\phi(d_{i,j},y) \leftrightarrow  (d'_{i,j} = d''_{i,j}))$.
  Then $d := ((d_{i,j})_{ij},(d'_{i,j})_{ij},(d''_{i,j})_{ij})$ is as required:

  For each $a\in A$, $\psi(a,d)$ holds iff $|\set{i<n}{p_i(y)|_{D_i} \vdash \phi(a,y)}| > \frac{1}{2}n$. Thus, if $a \in A_0$ and $\phi(a,b)$ holds, then $|\set{i<n}{\phi(a,y) \in p_i}| > \frac{1}{2}n$, and for each such $i<n$, $p_i|_{D_i} \vdash \phi(a,y)$ (by choice of $D_i$), so $\psi(a,d)$ holds. On the other hand, if $\psi(a,d)$ holds, then clearly $|\set{i<n}{\phi(a,y) \in p_i}| > \frac{1}{2}n$, hence $\phi(a,b)$ holds.
\end{proof}


\begin{remark}
  In fact, by a Löwenheim-Skolem argument, to prove \cref{cor:luhds} we require \cref{thm:rounded average of compressible} only in the case that $A$ is countable. The proof of this case of \cref{thm:rounded average of compressible} is slightly simpler, in that we can use $\omega$ in place of $\pfin(\kappa)$ (using \cite[Lemma 2.8]{simon-invariant} instead of \cite[Lemma 2.9]{simon-invariant}), and the usual Ramsey theorem in place of \cref{prop:Ramsey for Pfin}.
\end{remark}

\begin{remark}
  If $\phi(x,y)$ is stable, one can use \cref{thm:rounded average of compressible} similarly to get a new way to see definability of $\phi$-types over arbitrary sets (since $k$-isolated types are definable).
\end{remark}

\subsection{Hypes}
\begin{definition}\label{def:hypes}
  Suppose $\phi(x,y)$ is a formula and $A \subseteq \mathcal{U}^y$. For $k\in \N$, a \defn{$k$-hype\footnote{The term ``$k$-hype'' stands for $k$-hypotype.} in $\phi$ over $A$}\footnote{As in \cref{sec:notations}, in this notation we have the partition in mind and $x$ is the first tuple in the partition, so we do not specify $x$.} is a collection $\Gamma$ of instances of $\phi$ and $\neg \phi$ over $A$ such that:
  \begin{enumerate}
    \item It is $k$-consistent: if $S\subseteq \Gamma$ is of size $\leq k$, then $S$ is consistent.
    \item For any $a \in A$, either $\phi(x,a) \in \Gamma$ or $\neg \phi(x,a) \in \Gamma$, but not both. 
  \end{enumerate}

  Suppose $\pi(x)$ is a (small) partial type. We say that $\Gamma$ is $k$-consistent with $\pi$ if in (1) we ask that $S \cup \pi$ is consistent.

  Let $S^x_{\phi,k}(A)$ be the set of $k$-hypes in $\phi$ over $A$, and $S^{\pi}_{\phi,k}(A)$ the set of $k$-hypes which are $k$-consistent with $\pi$.

  As with types, if $\Gamma$ is a $k$-hype, we use the notation $\Gamma|_{A'}$ for the restriction of $\Gamma$ to $A' \subseteq A$ with the obvious meaning.
\end{definition}

Suppose $\phi(x,y)$ is a formula, $\pi(y)$ a small partial type and $A\subseteq \U^x$. Let $M$ be a small model containing $A$ and the domain of $\pi$. We will consider the following auxiliary 2-sorted structure
\[\hat{A}=(A,S^{\pi}_{\phiopp,k}(A),R)\]
where $A$ and $S^{\pi}_{\phiopp,k}(A)$ are in distinct sorts $P,Q$ respectively and $R(x,z)$ is interpreted as $R(a,\Gamma) \Leftrightarrow \phi(a,y) \in \Gamma$ (so $R \subseteq P \times Q$). Let $\psi(x,z) = R(x,z)$.

\begin{lemma}\label{lem:k-hypes} Suppose $N \succ \hat{A}$ and $e \in Q^N$. Then:
  \begin{enumerate}
    \item If $\phi$ is NIP and $\vc(\phi) < k$, then $\psi$ is NIP and $\vc(\psi) \leq \vc(\phi)$. 
    \item Let $\Gamma_e = \set{\phi(a,y)^{N\vDash R(a,e)}}{a \in A}$. Then $\Gamma_e \in S^{\pi}_{\phiopp,k}(A)$.
    \item If $\tp_{\psi^{\opp}}(e/A) \in S_{\psi^{\opp}\down k}(A)$, then $\Gamma_e \in S^{\pi}_{\phiopp \down k}(A)$, i.e., $\Gamma_e$ is consistent with $\pi$ and moreover $k$-compressible modulo $\pi$.
  \end{enumerate}
\end{lemma}

\begin{proof}
  (1) Suppose $C\subseteq_{\fin} A$ is of size $|C| = \vc(\phi) + 1$, and suppose that for any subset $C' \subseteq C$ there is some $\Gamma_{C'} \in S^{\pi}_{\phiopp,k}(A)$ such that for all $c\in C$, $R(c,\Gamma_{C'})$ holds iff $c \in C'$. As $k\geq |C|$, for any $C' \subseteq C$, $\set{\phi(c,y)^{R(c,\Gamma_{C'})}}{c \in C}$ is consistent; let $e_{C'}$ realise it. Then we get that $M\vDash \phi(c,e_{C'})$ iff $R(c,\Gamma_{C'})$ iff $c \in C'$. Thus $C$ is shattered by $\phi(x,y)$ and has size $\vc(\phi)+1$, contradiction

  (2) Clearly (2) in \cref{def:hypes} holds. For (1), suppose $S= \Gamma_e|_{A_0}$ has size $k$. Then, as $N$ is an elementary extension of $\hat{A}$, there is some $\Gamma \in S^{\pi}_{\phiopp,k}(A)$ such that for all $a \in A_0$, $N\vDash R(a,e)$ iff $\hat{A} \vDash R(a,\Gamma)$. In particular, $S \subseteq \Gamma$. Since $\Gamma$ is $k$-consistent, $S$ is consistent.

  For (3), suppose $A_0 \subseteq_{\fin} A$. Let $A_1 \subseteq A$ be of size $k$ such that $\tp_{\psi^{\opp}}(e/A)|_{A_1} \vdash \tp_{\psi^{\opp}}(e/A)$. Since $A_0$ is arbitrary, it is enough to show (a) that there exists $d \vDash \Gamma_e|_{A_1} \cup \pi$ (in $\mathcal{U}$) and (b) that any such $d$ satisfies $\Gamma_e|_{A_0}$. (a) follows from (2). For (b), let $q = \tp_{\phiopp}(d/A)$, so in particular, $q \in  S^{\pi}_{\phiopp,k}(A)$. Note that in $\hat{A}$, $q \vdash \tp_{\psi^{\opp}}(e/A)|_{A_1}$. Hence $q \vdash \tp_{\psi^{\opp}}(e/A)|_{A_0}$. Hence we have that $\phi(a,y)^{R(a,e)} \in q$ for any $a \in A_0$, or in other words, $d \vDash \Gamma_e|_{A_0} \cup \pi$.
\end{proof}

\begin{remark} \label{rem:qq theorem implies hypes are covered}
  Suppose $\phi(x,y)$ is NIP and $k > \vc(\phi)$. Then by the $(p,q)$-theorem (\cref{fac:(p q) theorem}) and compactness, for any $k$-hype $\Gamma$, there are $N:=\Npq(k,k)$ types $p_0, \ldots, p_{N-1}$ such that $\Gamma \subseteq \bigcup_{i<N}p_i$.
\end{remark}

The following extends \cref{rem:qq theorem implies hypes are covered}: not only are hypes covered by types, they are the rounded average of (compressible) types. It also generalises \cref{thm:rounded average of compressible} to hypes.

\begin{theorem} \label{thm:k-hypes majority average}
  Let $d \in \N$ and $\alpha \in (1/2,1]$.
  Then there exist $n$ and $k$ depending only on $d, \alpha$ such that the following holds.
  If $\phi(x,y)$ is a formula such that $\vc(\phi)\leq d$, then
  for any $A \subseteq  \U^x$ and a (small) partial type $\pi(y)$,
  any $k$-hype $\Gamma \in S_{\phiopp,k}^\pi(A)$ is the $\alpha$-rounded average of $n$ types in $S_{\phiopp\down k}^\pi(A)$.

  In fact, $k$ and $n$ can be taken to be the numbers from \cref{thm:rounded average of compressible}.
\end{theorem}

\begin{proof}
  Let $k$ and $n$ be as in \cref{thm:rounded average of compressible}, and note that by \cref{rem:Bvc(n k) >= n k}, we have that $k > d$. 
  Let $\phi(x,y)$ be such that $\vc(\phi)\leq d <k$. Let $\Gamma$ be a $k$-hype in $\phiopp$ over $A$.

  Consider the structure $\hat{A}$ as above, and let $\psi(x,z) = R(x,z)$. By \cref{lem:k-hypes}(1), as $\vc(\phi) < k$, $\psi(x,z)$ is NIP and $\vc(\psi)\leq \vc(\phi)$. By \cref{thm:rounded average of compressible}, any $\psi^{\opp}$-type is the $\alpha$-rounded average of $n$ types in $S_{\psi^{\opp}\down k}(A)$. In particular, this is true for $\tp_{\psi^{\opp}}(\Gamma/A)$; let $r_0,\ldots, r_{n-1}$ witness this. Let $e_i \vDash r_i$ for $i<n$ (in an elementary extension $N\succ \hat{A}$), and let $\Gamma_i = \Gamma_{e_i}$ be the corresponding $k$-hypes as in \cref{lem:k-hypes}(2). By \cref{lem:k-hypes}(3), we get that $\Gamma_i$ is a type in $S^{\pi}_{\phiopp\down k}(A)$.

  Finally, since $\tp_{\psi^{\opp}}(\Gamma/A)$ is the $\alpha$-rounded average of $r_0,\ldots r_{n-1}$, it follows by definition that $\Gamma$ is the $\alpha$-rounded average of $\Gamma_0,\dots,\Gamma_{n-1}$.
\end{proof}

We deduce the existence of honest definitions for $k$-hypes.
 \begin{corollary}\label{cor:honest definitions for k-hypes}
  Let $\phi(x,y)$ be NIP and let $k$ be as in \cref{thm:k-hypes majority average} for $d=\vc(\phi)$ and $\alpha = 1/2$.
  Then there exists $\psi(x,z)$ such that
  if $A \subseteq  \U^x$ with $|A| > 1$
  and $\Gamma \in S_{\phiopp,k}(A)$ is a $k$-hype,
  then $\psi(x,z)$ is an honest definition of $\Gamma$ in the sense that if $A_0 \subseteq_{\fin} A$, then there is some $d \in A^z$ such that:
  \begin{enumerate}
    \item If $a \in A_0$ and $\phi(a,y) \in \Gamma$ then $\psi(a,d)$ holds.
    \item For all $a \in A$, if $\psi(a,d)$ holds, then $\phi(a,y) \in \Gamma$.
  \end{enumerate}

  Namely, \[\psi(x,(\zz,\zz',\zz'')) :=
  \Maj_{i<n} \forall y \left({\bigwedge_{j<k}
  (\phi(z_{i,j},y) \leftrightarrow  (z'_{i,j} = z''_{i,j})) \rightarrow  \phi(x,y)}\right),\]
  where $n$ and $k$ are as in \cref{thm:k-hypes majority average}.
\end{corollary}
\begin{proof}
   The proof is the same as the one of \cref{cor:luhds}, using \cref{thm:k-hypes majority average}.
\end{proof}

We relate hypes to the Shelah expansion which we now recall.

\begin{definition}
  For a structure $M$, the Shelah expansion $M^{\Sh}$ of $M$ is given by: for any formula $\phi(x,y)$ and any $b \in \U^y$, add a new relation $R_{\phi(x,b)}(x)$ interpreted as $\phi(M,b)$.
\end{definition}

\begin{fact} \cite{MR2570659} \label{fac:Shelah expansion}
  If $T$ is NIP then for any $M\vDash T$, $M^{\Sh}$ is NIP.
\end{fact}

\begin{corollary}
  Suppose $T$ is NIP, and let $M \vDash T$.
  For each formula $\phi(x,y)$, let $k_\phi$ be as in \cref{thm:k-hypes majority average} for $d=\vc(\phi)$ and $\alpha = 1/2$.

  Consider the expansion $M^{h\Sh}$ of $M$ given by naming for each partitioned $\L$-formula $\phi$ and each $k_\phi$-hype $\Gamma \in S_{\phiopp,k_{\phi}}(M^x)$ the set $R_\Gamma := \set{a\in M^x}{\phi(a,y)\in \Gamma}$. Then $M^{h\Sh}$ is interdefinable with $M^{\Sh}$ and in particular is NIP.
\end{corollary}

\begin{proof}
  Since every $\phiopp$-type is in particular a $k_\phi$-hype, every $R_{\phi(x,b)}(M)$ is definable in $M^{h\Sh}$.

  For the other direction, fix some formula $\phi(x,y)$ and some $k_\phi$-hype $\Gamma \in S_{\phiopp,k}(M)$. By \cref{thm:k-hypes majority average} (applied with $d=\vc(\phi)$ and $\alpha = 1/2$), there are $r_0,\dots,r_{n-1} \in S_{\phiopp}(M)$ such that $\Gamma$ is the rounded average of $r_0,\dots,r_{n-1}$. Let $a_0 \vDash r_0,\dots, a_{n-1}\vDash r_{n-1}$. Then $R_{\Gamma}(M) = \Maj_{i<n}(R_{\phi(x,a_i)}(M))$, and thus is definable in $M^{\Sh}$.
\end{proof}

\subsection{UDTFS for pseudofinite types}

Here we extend UDTFS to \emph{pseudofinite} types: every pseudofinite $\phi$-type (\cref{def:pseudofinite type}) is definable, and uniformly so.

\begin{definition}\label{def:pseudofinite type}
  Let $\L' = \L \cup \{P,Q\}$ where $P,Q$ are predicates for subsets of $\U^x$.

  Suppose $\phi(x,y)$ is an $\L$-formula, $M\vDash T$, $D\subseteq M^x$, and $p \in S_{\phiopp}(D)$. For $\epsilon<2$, let $D^{\epsilon} = \set{a\in D}{\phi(a,y)^\epsilon \in p}$. Then $p$ is \defn{pseudofinite} if for every $\L'$-sentence $\varphi$, if $(M,D^0,D^1) \vDash \varphi$ then there is an $\L'$-structure $N$ such that $N \vDash \varphi$ and $P^N,Q^N$ are finite.
\end{definition}

\begin{remark}
  In the notation of \cref{def:pseudofinite type}, a type $p\in S_{\phiopp}(D)$ is pseudofinite iff there is a model $(N,E^0,E^1)\equiv (M,D^0,D^1)$ which is an ultraproduct $\prod_{i\in I} N_i/U$ such that $(E^\epsilon)^{N_i}$ is finite for each $i \in I$ and $\epsilon<2$ (essentially the same proof as in \cite[Lemma~1]{MR2007169} works).
\end{remark}

\begin{theorem}\label{thm:pseudofinite types are definable}
  Suppose $\phi(x,y)$ is NIP. Then there is a formula $\psi(x,z)$ such that whenever $M\vDash T$, $D\subseteq M^x$ is of size $>1$ and $p \in S_{\phiopp}(D)$ is pseudofinite, $p$ is definable by an instance of $\psi$ over $D^z$.

  Moreover, if $T$ has Skolem functions, then we can choose $\psi(x,z)$ to be NIP.
\end{theorem}

\begin{proof}
  For the first part, let $k$ and $\psi(x,z)$ be as in \cref{cor:honest definitions for k-hypes}. Suppose that $p(y)$ is not definable by an instance of $\psi$ over $D^z$. Working in the expansion $(M,D^0,D^1)$ as in \cref{def:pseudofinite type}, we get that in some $\L'$-structure $(N,E^0,E^1)$, letting $E = E^0 \cup E^1$, the following hold:
  \begin{itemize}
    \item $E$ is finite of size $>1$.
    \item The formula $\phi(x,y)$ is NIP in $N$ and its VC-dimension equals $\vc(\phi)$.
    \item The set of formulas $\Gamma = \set{\phi(a,y)}{a\in E^1} \cup \set{\neg \phi(a,y)}{a \in E^0}$ is a $k$-hype.
    \item $E^1$ is not definable in $N$ by any instance of $\psi$ over $E^z$.
  \end{itemize}
  However, by the choice of $\psi$ and as $E$ is finite, there is some $d \in E^z$ such that $\psi(E,d) = E^1$, contradiction.

  For the second part, assuming that $T$ has Skolem functions, we let
  \[\psi(x,(\zz,\zz',\zz'')) :=
  \Maj_{i<n} \left(\phi(x,f(z_i,\zz'_i,\zz''_i))\right),\]
  where $n$ is as in \cref{thm:k-hypes majority average} and $f$ is a $\emptyset$-definable function such that $T$ thinks that if $\exists y \bigwedge_{j<k}(\phi(z_{i,j},y) \leftrightarrow  (z'_{i,j} = z''_{i,j}))$ then $f(\zz_i,\zz'_i,\zz''_i) \vDash \bigwedge_{j<k}(\phi(z_{i,j},y) \leftrightarrow  (z'_{i,j} = z''_{i,j}))$ (whose existence we assumed).

  Note that $\psi(x,z)$ is NIP, since $\phi(x,f(z))$ is NIP; see also \cite[Proof of Proposition 26]{UDTFS}. To see that it works, assume not. Then using the same argument as above, we get an $\L'$-structure $N$ with the same properties as above. Now, review the proof of \cref{cor:luhds}. When the domain $D$ of the $k$-compressible types $p_i$ (for $i<n$) is finite, then $p_i$ is in fact isolated by a conjunction of $k$ instances of $\phiopp$ or its negation, thus, putting the isolating parameters for $z_i$ and coding the negations using $z_i',z_i''$ and two elements from $D^N$, we are done.
\end{proof}

\begin{remark}
  Note that if $p$ is realised in $M$, then \cref{thm:pseudofinite types are definable} follows directly from UDTFS: in that case, in the proof one can replace the demand about $\Gamma$ being a hype by it being a type. 
\end{remark}

\begin{remark}
  Clearly \cref{thm:pseudofinite types are definable} implies UDTFS (\cref{cor:UDTFS}), and hence its conclusion implies that $\phi(x,y)$ is NIP (see e.g.\ the proof of Theorem~14 in \cite{UDTFS}).
\end{remark}

\section{Compressibility as an isolation notion}\label{sec:compressiblity as an isolation}
In this section we study properties of compressibility seen as an isolation notion (mostly) under NIP, and in particular as a way to construct models analogous to constructible models in totally transcendental theories. Towards that we prove a transitivity result for compressibility in \cref{prop:transitivity}, which uses the type decomposition theorem from \cite{simon-decomposition}.

As an application, we will show that if $T$ is unstable and $M\vDash T$ is $\omega$-saturated, then there are arbitrarily large elementary extensions $N$ of $M$ such that every generically stable type over $M$ (see \cref{def:GS types over models}) realised in $N$ is realised in $M$ (this is \cref{{cor:extend Preserve Stable},{rem:extend preserve stable omega-sat}}).

\subsection{Monotonicity}
\begin{lemma} \label{lem:monotonicity}
  Suppose $b,c$ are finite tuples. If $\tp(cb/A)$ is compressible then so are $\tp(c/Ab)$ and $\tp(b/A)$.
\end{lemma}
\begin{proof}
  We start by showing that $\tp(c/Ab)$ is compressible. Suppose $z$ is a tuple of variables such that $b\in \U^z$. Given a formula $\phi(x,y)$, let $\Phi$ be the set of all formulas of the form $\psi(x,z,y)$ we get from substituting variables from $y$ by variables from $yz$ in $\phi$.
  Fix some formula $\psi(x,z,y)\in \Phi$. By assumption, there is some formula $\zeta_\psi(xz,w)$ that compresses $\tp_\psi(cb/A)$ (with the partition $\psi(xz,y)$). This means that for any  $A_0 \subseteq_{\fin}  A$ there is some $a_{\psi,A_0} \in A^w$ such that $\tp(cb/A) \vdash \zeta_\psi(xz,a_{\psi,A_0}) \vdash  \tp_\psi(cb/A_0)$.

  Then we have that $\tp(c/Ab) \vdash \bigwedge_{\psi\in \Phi}\zeta_\psi(x,ba_{\psi,A_0}) \vdash  \tp_\phi(c/A_0b)$, and this shows that $\tp(c/Ab)$ is compressible.

  Now we show that $\tp(b/A)$ is compressible. Fix some formula $\phi(z,y)$. Suppose $\zeta(xz,w)$ compresses $\tp_{\phi}(cb/A)$. Fix some $A_0 \subseteq_{\fin} A$ and suppose $a \in A^w$ is such that $\tp(cb/A) \vdash \zeta(xy,a) \vdash  \tp_\phi(cb/A_0)$.
  Note that if $a_0\in A_0^y$ is such that $\vDash  \forall x,z (\zeta(x,z,a) \rightarrow  \phi(z,a_0)^{\epsilon})$ for $\epsilon<2$,
  then $\vDash  \forall z (\exists x \zeta(x,z,a) \rightarrow  \phi(z,a_0)^{\epsilon})$,
  so $\tp(b/A) \vdash \exists x \zeta(x,z,a) \vdash  \tp_\phi(b/A_0)$.
\end{proof}

\begin{remark} \label{rem:converse to monotonicity}
  We cannot hope for \cref{lem:monotonicity} to hold when $b$ is infinite: every type over $\emptyset$ is compressible trivially, so if $\tp(a/B)$ is not compressible and $b$ enumerates $B$, then $\tp(ab)$ is compressible while $\tp(a/b)$ is not.

  However, the converse to \cref{lem:monotonicity} holds for infinite tuples as well (see \cref{rem:finite transitivity} below). For finite tuples this can be seen by a direct argument of this kind, but for infinite tuples we will need stronger tools which we will develop in the next section under NIP. 
\end{remark}


\begin{definition} \label{def:set compressible over}
  Suppose $B,A$ are sets. We say \defn{$B$ is compressible over $A$} if $\tp(\bar{b}/A)$ is compressible where $\bar{b}$ is some (any) tuple enumerating $B$. 
\end{definition}

\begin{remark}
  The set $B$ is compressible over $B$ (even isolated).
\end{remark}

\begin{remark} \label{rem:B compressible iff all finite subsets}
  By \cref{rem:compressible iff finite restrictions} $B$ is compressible over $A$ iff for every finite tuple $b$ from $B$, $\tp(b/A)$ is compressible over $A$.
\end{remark}

\begin{lemma} \label{lem:set + one more comp over set}
  Given a set $B$ and a tuple $a$,
  the set $Ba$ is compressible over $B$ if and only if $\tp(a/B)$ is compressible.
\end{lemma}
\begin{proof}
  Left to right follows from \cref{lem:monotonicity}, so suppose that $\tp(a/B)$ is compressible. Let $b$ be a tuple enumerating $B$. We must show that $\tp(ab/B)$ is compressible. By \cref{rem:compressible iff finite restrictions}, it is enough to prove that $\tp(ab'/B)$ is compressible where $b'$ is a finite sub-tuple of $b$. Let $y$ be a tuple of variables in the sort of $b'$. Let $\phi(x,y,z)$ be a formula and let $\psi(x,s)$ compress $\tp_{\phi(x,yz)}(a/B)$. Then $\psi(x,s) \land y=t$ compresses $\tp_{\phi(xy,z)}(ab'/B)$.
\end{proof}

%
%
%

\subsection{Type decomposition and rescoping compressibility}\label{sec:decompComp}

Here we use the results from \cite{simon-decomposition} to prove that compressibility can be \emph{rescoped} to an arbitrary subset of the domain (see \cref{{prop:compressible iff orth to MS},{prop:comp over smaller set with constants}}).

For the remainder of \cref{sec:compressiblity as an isolation} \textbf{we assume that $T$ is NIP} unless otherwise specified.

We first recall the definition of a generically stable partial type. As opposed to previous sections, here a partial type does \underline{not} have to be small, i.e., it is over $\U$. We call such partial types \defn{global partial types}. As for global types, a global partial type $\pi$ is $A$-invariant if it is invariant under automorphisms of $\U$ fixing $A$.

\begin{remark} \label{rem:extension of partial global type}
  Suppose $\pi(x)$ is a global partial type. Then for any small set $A$, if $a \vDash \pi|_A$, then $\pi(x) \cup \tp(a/A)$ is consistent, and hence for any $B$ there is some $a' \equiv_A a$ such that $a' \vDash \pi|_B$.
\end{remark}

\begin{definition}
  We say that a global partial type $\pi$ is \defn{ind-definable} over $A$ if for every $\phi(x,y)$, the set $\set{b\in \U^y}{\phi(x,b) \in \pi}$ is ind-definable over $A$, i.e., it is a union of $A$-definable sets.
\end{definition}

\begin{remark} \cite[Discussion after Definition~2.1]{simon-decomposition}
  Note that $\pi(x)$ is ind-definable iff $\set{\phi(x,c)}{\pi\vdash \phi(x,c)}$ is ind-definable.
\end{remark}

\begin{fact}\cite[Lemma~2.2]{simon-decomposition} \label{fac:ind-def iff realising over is tp-def}
  Let $\pi(x)$ be an $A$-invariant global partial type. Then $\pi$ is ind-definable over $A$ if and only if the set $X=\set{(a,\bar b)}{\bar b\in \U^{\omega}, a\vDash \pi|_{A\bar b}}$ is type-definable over $A$.
\end{fact}

\begin{definition} \label{def:GS partial type}
  Let $\pi(x)$ be a global partial type. We say that $\pi$ is \defn{generically stable over $A$} if $\pi$ is ind-definable over $A$ and the following holds:

  \hypertarget{GS}{(GS)} if $\varsequence{a_k}{k<\omega}$ is such that $a_k \vDash \pi|_{Aa_{<k}}$ and $\pi \vdash \phi(x,b)$, then for all but finitely many values of $k$ we have $\U \vDash \phi(a_k,b)$.
\end{definition}

\begin{remark}
  Note that a global type $p(x) \in S^x(\U)$ is generically stable over $A$ as in \cref{def:GS global type} iff it is generically stable over $A$ as a partial type (note that it is $A$-definable by \cref{fac:generically stable -> dfs}).
\end{remark}

\begin{remark}
  Much like in \cref{rem:compressible iff finite restrictions}, a global partial type $\pi(x)$ is generically stable iff its restriction to any finite sub-tuple $x'$ of $x$ is generically stable. (Note that we do not assume that $\pi$ is ind-definable.)

  Why? Clearly if the restrictions are all generically stable then $\pi$ is, so we show the converse. Assume that $\pi$ is generically stable and fix some finite $x' \subseteq x$. First note that $\pi \restriction x'$ is ind-definable, so we show \hyperlink{GS}{(GS)}. Assume that $\varsequence{a_k'}{k<\omega}$ is such that $a_k' \vDash (\pi\restriction x')|_{Aa'_{<k}}$ and $\pi \vdash \phi(x',b)$. By induction on $n<\omega$ we construct sequences $\varsequence{a_i}{i<n}$ such that $a_i \restriction x' = a_i'$ and $a_i \vDash \pi|_{Aa_{<i}}$ for all $i<n$. Suppose we found such a sequence $\varsequence{a_i}{i<n}$. Since $a_n' \vDash (\pi\restriction x')|_{Aa'_{<k}}$, as in \cref{rem:extension of partial global type} there is some $a_n'' \vDash \pi|_{Aa_{<n}}$ such that $(a_n''\restriction x') \equiv _{{Aa'_{<n}}} a_n'$. Let $\sigma$ be an automorphism fixing ${Aa'_{<n}}$ such that $\sigma(a_n''\restriction x') = a_n'$. Then $\varsequence{\sigma(a_i)}{i<n}\sigma(a_n'')$ is a sequence of length $n+1$ which is as required. By compactness and \cref{fac:ind-def iff realising over is tp-def}, there is some sequence $\varsequence{a_i}{i<\omega}$ such that $a_i \restriction x' = a_i'$ and $a_i \vDash \pi|_{Aa_{<i}}$ for all $i<\omega$. By \hyperlink{GS}{(GS)} for $\pi$, for all but finitely many values of $k$ we have $\vDash \phi(a_k',b)$, as required. 
\end{remark}

\begin{definition}
  We say that a global partial type $\pi(x)$ is \defn{finitely satisfiable in $A \subseteq \U$} if any formula implied by $\pi$ has a realisation in $A$.
\end{definition}

\begin{fact}\label{fac:GS and non-forking}
  Let $\pi(x)$ be a global partial type generically stable over $A$. Then:

  \hypertarget{FS}{(FS)} $\pi$ is finitely satisfiable in every model containing $A$.

  \hypertarget{NF}{(NF)} Let $\phi(x,b)$ be such that $\pi \vdash \phi(x,b)$ and take $a\vDash \pi|_A$ such that $\vDash \neg \phi(a,b)$. Then both $\tp(b/Aa)$ and $\tp(a/Ab)$ fork over $A$.
\end{fact}

We now state \cite[Theorem~4.1]{simon-decomposition} in the form we will use it below. Our formulation follows from the proof (rather than the statement) of \cite[Theorem~4.1]{simon-decomposition}, in particular from \cite[Proposition~4.7]{simon-decomposition}.

\begin{fact} \cite[Proposition~4.7]{simon-decomposition} \label{fac:decomp}
  Given a type $\tp(a/A)$ and $q \in \Sfs{A}{\U}$,
  there exists a global partial type $\pi(x)$ generically stable over $A$ with $a \vDash  \pi|_A$ such that if $(X,<)$ is an infinite linear order, $I \vDash  q^{(X)}|_{Aa}$,
  $b \vDash  q|_{AI}$ and $a \vDash  \pi|_{AIb}$, then $b \vDash  q|_{Aa}$. 
\end{fact}

\begin{remark} \label{rem:on decomposition}
  \begin{enumerate}[(i)]
    \item In \cite[Proposition~4.7]{simon-decomposition} the generically stable type constructed depends on $q$. Call it $\pi_q$. However, in the paragraph after \cite[Proposition~4.7]{simon-decomposition}, it is remarked that taking $\pi$ to be the union of all the $\pi_q$ works for all $q$.
    \item It is not assumed in \cite{simon-decomposition} that the the tuple $a$ above is finite.
    \item Throughout the proof of \cite[Theorem~4.1]{simon-decomposition}, the sequences are assumed to be densely ordered without endpoints. In particular, that is the case for $I$ above. However, the result is true for any infinite $I$. Indeed, suppose $I,a,A,b$ are as in \cref{fac:decomp}. By Ramsey and compactness (and as $I$ is infinite) there is an $Aab$-indiscernible sequence $I' = \sequence{a_i}{i \in \Q}$ realising the EM-type of $I$ over $Aab$. Since $I$ is $Aa$-indiscernible, it follows that $I' \vDash q^{(\Q)}|_{Aa}$, and since $q$ is $A$-invariant, $b \vDash q|_{AI'}$. Also, by \cref{fac:ind-def iff realising over is tp-def}, $a \vDash \pi|_{AI'b}$. So $I'$ satisfies all the requirements of \cref{fac:decomp} and is densely ordered with no endpoints, so $b \vdash  q|_{Aa}$ as required.
    \item In the context of \cref{fac:decomp}, it follows (by applying an automorphism; note that both $q$ and $\pi$ are $A$-invariant) that if $a' \vDash  \pi|_{AIb} \cup \tp(a/AI)$, then $b \vDash  q|_{Aa'}$.
   \end{enumerate}
\end{remark}



\begin{proposition} \label{prop:compressible iff orth to MS}
  Let $(X,<)$ be any (small) infinite linearly ordered set. A type $\tp(a/A)$ is compressible if and only if
  for every $q \in \Sfs{A}{\U}$,
  if $I \vDash  q^{(X)}|_{Aa}$,
  then \[q|_{AI} \vdash  q|_{Aa}.\]
  Moreover, if $X$ has no first element then $q|_{AI} \vdash  q|_{AIa}$.
\end{proposition}
\begin{proof}
  That this condition implies compressibility follows from \cref{fac:compressible iff weakly orthogonal}(3$\Rightarrow$1), since $I$ can be taken from $A'$ (where $(A',a)$ is saturated enough).

  For the converse, assume compressibility and let $q, I$ be as in the proposition. Let $d$ be given by \cref{cor:orthFSComp} so that $q|_{Ad} \vdash  q|_{Aa}$ and $\tp(d/Aa)$ is finitely satisfiable in $A$. By perhaps changing $d$ (by applying an automorphism fixing $Aa$), we may assume that $I \vDash q^{(X)}|_{Aad}$.

  Let $b \vDash  q|_{AI}$. We want to show that $b \vDash q|_{Aa}$.

  Applying \cref{fac:decomp} (and \cref{rem:on decomposition}(ii),(iv)) to $\tp(d/A)$, we obtain a generically stable over $A$ global partial type $\pi$ with $d \vDash  \pi|_A$ such that $b \vDash  q|_{Ad'}$ for any $d' \vDash  \pi|_{AIb} \cup \tp(d/AI)$.

  Now $\tp(d/Aa)$ is finitely satisfiable in, and hence does not fork over, $A$. Similarly, $\tp(I/Aad)$ is finitely satisfiable in $A$ and so does not fork over $Aa$. By applying \hyperlink{NF}{(NF)} twice, it follows that $d \vDash  \pi|_{AIa}$.

  Hence $\pi' := \pi \cup \tp(d/AIa)$ is consistent by \cref{rem:extension of partial global type}.
  Let $\kappa := |\L(AIb)|^+$,
  and let $(d_i)_{i < \kappa}$ be a Morley sequence in $\pi'$ over $AIa$, i.e., $d_i \vDash  \pi'|_{AIad_{<i}}$.
  By \hyperlink{GS}{(GS)} and the choice of $\kappa$, for some $i<\kappa$ we have $d_i \vDash  \pi|_{AIb}$.
  Then $d_i \vDash  \pi|_{AIb} \cup \tp(d/AI)$, so $b \vDash  q|_{Ad_i}$.
  But $d_i \equiv _{Aa} d$, so $q|_{Ad_i} \vdash  q|_{Aa}$.
  So $b \vDash  q|_{Aa}$, as required.

  We conclude the ``moreover'' part. Suppose that $X$ has no first element and that $I\vDash q^{(X)}|_{Aa}$, $b\vDash q|_{AI}$. By compactness we can find some $I'$ of order type $\omega \times (X+1)$ (where $X+1$ is adding one more element in the end of $X$ and the product is ordered lexicographically), such that $I' + I$ is indiscernible over $Aa$ and over $Ab$. Thus, partitioning $I'$ into $(X+1)$-sequences, we have that $I' \vDash (q^{(X+1)})^{(\omega)}|_{Aa}$, and $(I+b) \vDash q^{(X+1)}|_{AI'}$. Applying the first part to $q^{(X+1)}$, we have that $(I+b) \vDash q^{(X+1)}|_{Aa}$. It follows that $b \vDash q|_{AIa}$ as required.
\end{proof}

The following corollary will not be used in this paper.
\begin{corollary} \label{cor:generically co-distal}
  A type $p=\tp(a/A) \in S(A)$ is compressible iff for any $q \in \Sfs{A}{\U}$ and any Morley sequence $I:=I_1 + (b) + I_2$ of $q$ over $A$ where $I_1$ has no first element, if $I_1 + I_2$ is a Morley sequence of $q$ over $Aa$ then so is $I$.
\end{corollary}
\begin{proof}
  Right to left is clear by \cref{prop:compressible iff orth to MS}, so suppose that $p$ is compressible, and we are given $I$ as above. Let $\varsequence{b_i}{1\leq i<n}$ be some finite subsequence from $I_2$ and let $b_0 = b$. Then by applying the ``moreover'' part of \cref{prop:compressible iff orth to MS} inductively, $b_i \vDash q|_{AI_1b_{<i}a}$. Since this is true for any $n<\omega$, $I$ is $Aa$-indiscernible.
\end{proof}


\begin{remark}
  One might call the condition in \cref{cor:generically co-distal} \defn{generic co-distality}: it is co-distality in a generic sense. For a definition of distal and co-distal types and a short discussion, see \cite[Definition~4.21 and Remark~4.22]{MR4216280}.
\end{remark}

\begin{definition}
  Suppose $B,A \subseteq \U$ are (small) sets. Say that a type $p \in S(A)$ \defn{is compressible up to $B$} if $p$ is compressible in $\U_B$ (in the language $\L(B)$): in \cref{def:global compression}, all the formulas are over $B$.
\end{definition}

\begin{remark} \label{rem:comp up to the difference}
  Suppose that $A \subseteq B$, $p \in S(A)$ and $p$ is compressible up to $B$. Then $p$ is compressible up to $C:=B\setminus A$: given a formula $\phi(x,y)$ over $C$, there is an $\L$-formula $\psi(x,z,w,t)$ and $a \in A^w,c\in C^t$ such that $\psi(x,z,ac)$ compresses $p\restriction \phi$. But then $\psi(x,z,w,c)$ compresses $p \restriction \phi$.
\end{remark}

\begin{proposition} \label{prop:comp over smaller set with constants}
  If a type $\tp(a/B)$ is compressible and $A \subseteq  B$,
  then $\tp(a/A)$ is compressible up to $B$.
\end{proposition}
\begin{proof}
  We use \cref{prop:compressible iff orth to MS}, so we are given $q \in \Sfs{A}{\U_B}$, and we need to show that if $I \vDash q^{(\omega)}|_{Aa}$ then $q|_{AI} \vdash q|_{Aa}$, all working in $\U_B$. So suppose that in $\U_B$, $b \vDash q|_{AI}$ and we need to show that $b \vDash q|_{Aa}$.

  Taking the reduct to $\L$ (and identifying $q \restriction \L$ with $q$), $q \in \Sfs{A}{\U}$, $I \vDash q^{(\omega)}|_{Ba}$, $b \vDash q|_{BI}$ and we need to show that $b \vDash q|_{Ba}$. As $\tp(a/B)$ is compressible, \cref{prop:compressible iff orth to MS} implies exactly that, and we are done.
\end{proof}

We can now generalise \cref{cor:density of comp in countable NIP} to uncountable theories induced by adding constants to countable NIP theories (except that in the final clause we not obtain strength of the compressibility).


\begin{corollary} \label{cor:naming constants still comp}
  Suppose $T$ is countable and let $B \subseteq \U$.

  Suppose $A \subseteq \U$ is a set of parameters and $x$ is a countable tuple of variables. Then, compressible types are dense in $S^x(A)$ in $\Th(\U_B)$:

  Working in $\U_B$, if $\theta(x)$ is a consistent ($\L_B$-)formula over $A$,
  then there exists a compressible type $p(x) \in S(A)$ with $p(x) \vdash  \theta(x)$.

  More generally, if, working in $\U_B$, $\pi$ is a t-compressible partial type over $A$, then there exists a compressible $p \in S(A)$ with $\pi \subseteq  p$.
\end{corollary}
\begin{proof}
  Clearly it is enough to prove the ``more generally'' part.

  Note first that $\pi$ is t-compressible in $\U$ with respect to $AB$: if $\zeta(x,z,w)$ is such that $\zeta(x,z,b)$ compresses $\pi$ within $\pi$ in $\U_B$ (with respect to $A$) and $b \in B^w$, then $\zeta(x,z,w)$ compresses $\pi$ within $\pi$ with respect to $AB$. Indeed, given any $A_0B_0\subseteq_{\fin} AB$ there is some $d \in A^z$ such that (working in $\U$)
  \[\pi \vdash \phi(x,d,b) \vdash \pi|_{A_0B} \vdash \pi|_{A_0B_0}.\]

  By \cref{cor:density of comp in countable NIP}, there is a compressible type $p \in S(AB)$ containing $\pi$. By \cref{prop:comp over smaller set with constants}, $p$ is compressible up to $AB$, which implies (by \cref{rem:comp up to the difference}) that $p$ is compressible up to $B$, meaning that in $\U_B$, $p|_A \in S(A)$ (which equals $p$) is compressible.
\end{proof}

\subsection{Transitivity} \label{sec:transitivity}
We continue to assume that $T$ is NIP.

\begin{proposition} \label{prop:transitivity}
  Suppose $A \subseteq  B \subseteq  C$,
  $C$ is compressible over $B$, and $B$ is compressible over $A$.
  Then $C$ is compressible over $A$ (recall \cref{def:set compressible over}).
\end{proposition}
\begin{proof}
  By \cref{rem:compressible iff finite restrictions}, it is enough to show that $\tp(c/A)$ is compressible for any finite tuple $c$ from $C$.

  By \cref{prop:comp over smaller set with constants}, $\tp(c/A)$ is compressible up to $B$.

  So given $\phi(x,y) \in \L$ (where $c$ is of the sort of $x$),
  we get $\zeta(w,x,z) \in \L$, $b \in B^w$ such that for $A_0 \subseteq_{\fin}  A$
  there is $a \in A^z$ such that
  \[\tp(c/B) \vdash \zeta(b,x,a) \vdash  \tp_{\phi(x,y)}(c/A_0).\]
  Since $\tp(b/A)$ is compressible, for each $\epsilon<2$ there are $\xi_\epsilon(w,z_\epsilon)$ and $a_\epsilon \in A^{z_\epsilon}$ such that
  \[ \tp(b/A) \vdash \xi_\epsilon(w,a_\epsilon) \vdash \tp_{\forall x (\zeta(w,x,z) \rightarrow  \phi(x,y)^\epsilon)}(b/A_0a). \]
  Then
  \[\tp(c/A) \vdash (\exists w (\zeta(w,x,a) \wedge \xi_0(w,a_0) \wedge \xi_1(w,a_1))) \vdash \tp_{\phi(x,y)}(c/A_0). \]
  Hence $\tp(c/A)$ is compressible.
\end{proof}

\begin{example}\label{exa:transitivity is false for ABA}
  \cref{prop:transitivity} is false without NIP. For example, let $T$ be the theory of the countable atomless Boolean algebra. Let $A$ be a countable set of pairwise disjoint elements $\set{a_i}{i<\omega}$. For $i<\omega$, let $b_i = \bigcup a_{<i}$, and let $B = A \cup \set{b_i}{i<\omega}$. Let $c\neq 1\in \U$ contain all the elements from $A$. Let $C = Bc$. Then $C$ is compressible over $B$, $B$ is compressible over $A$, but $C$ is not compressible over $A$.

  Indeed, to show the first statement, it is enough to see
  that $p:=\tp(c/B)$ is compressible by \cref{lem:set + one more comp over set}. This is true since $\zeta(x,z):=(z<x) \land (x \neq 1)$ compresses it: for every finite $B_0 \subseteq B$, let $i<\omega$ be such that $\bigcup B_0 \leq b_i$. Then $p\vdash (b_i < x) \land (x \neq 1) \vdash p|_{B_0}$ by quantifier elimination.

  Since every tuple from $B$ is in the definable closure of $A$, $B$ is compressible over $A$ (every finite tuple is isolated).

  Finally, $C$ is not compressible over $A$, since compressibility is monotonic (\cref{lem:monotonicity}) and $c$ is not compressible over $A$. Why? Suppose $\psi(x,z)$ compresses $\tp_{x\geq y}(c/A)$. Let $A_0 = a_{<|z|+2} \subseteq A$, and assume $d\in A^z$ is such that $\psi(c,d)$ holds and $\psi(x,d)\vdash b_{|z|+2} \leq x$. Let $i<\omega$ be such that $a_i \notin d$. Then $(\bigcup d < x) \land (x\neq 1) \vdash \psi(x,d)$ (since this implies $\tp(c/d)$), so $(\bigcup d \cup a_i) \vDash \psi(x,d)$ but $\neg (b_{|z|+2} \leq (\bigcup d \cup a_i))$ holds, contradiction.

  Note that this example shows that in $T$, weak compressibility is different from compressibility (see \cref{que:weak compressibility} for the definition). Namely, $\tp(c/A)$ is weakly compressible (as witnessed by $\zeta(x,z)$) but not compressible.
\end{example}

\begin{remark} \label{rem:finite transitivity}
  The following rephrasing of \cref{prop:transitivity} is worth mentioning explicitly:
  given (perhaps infinite) tuples $c,b$ and a set $A$,
  if $\tp(c/Ab)$ and $\tp(b/A)$ are compressible, then $\tp(cb/A)$ is compressible.
  In this phrasing, this is a converse to \cref{lem:monotonicity} (see \cref{rem:converse to monotonicity}).

  This follows from \cref{prop:transitivity} and  \cref{lem:set + one more comp over set}. Note that \cref{prop:comp over smaller set with constants} where $B\setminus A$ is finite can be seen with a direct argument (not using NIP), so for finite tuples $c,b$, the above can be easily proven and does not require NIP.
\end{remark}

\subsection{Compressible models and applications}

In this section, $T'$ is a countable NIP theory with monster model $\U$ in the language $\L'$, $F\subseteq \U$ is some small subset, and $T = \Th(\U_F)$ in the language $\L:=\L'(F)$. In other words, $T$ is a complete theory we get by naming constants in a countable NIP theory (whose monster model is still denoted by $\U$, abusing notation).

%
%

\begin{definition}
  Say $B$ is \defn{compressibly constructible over $A$} if $B$ can be enumerated as $B = (b_i)_{i<\alpha}$ for some ordinal $\alpha$, such that $\tp(b_i / Ab_{<i})$ is compressible for all $i<\alpha$.
\end{definition}

As with other isolation notions, the existence of compressibly constructible models follows straightforwardly from density. In fact this is an instance of the abstract result \cite[Theorem~IV.3.1(5)]{Sh-CT}, but we give the proof.

\begin{proposition} \label{prop:comp constructible}
  For any set $A$, there exists a model $M \supseteq A$ which is compressibly constructible over $A$ and of cardinality $\leq |A|+|T|$.

  Moreover, if $B$ is compressibly constructible over $A$ then there is some model $M \supseteq B$ which is compressibly constructible over $A$ and of cardinality $\leq |B|+|T|$.
\end{proposition}
\begin{proof}
  Since $A$ is compressibly constructible over $A$, it is enough to prove the ``moreover'' part.

  Let $\lambda = |B|+|T|$.
  Construct an increasing chain $(B_i)_{i<\omega}$ as follows. Let $B_0 = B$, and given $B_i$, let $\sequence{\theta^i_j}{j<\lambda}$ enumerate all consistent formulas over $B_i$. Construct a sequence $\sequence{b^i_j}{j<\lambda}$ inductively by letting $b^i_j \vDash \theta^i_j$ be such that $\tp(b^i_j/B_ib^i_{<j})$ is compressible, using \cref{cor:naming constants still comp} (with $B=F$). Let $B_{i+1} = B_i \cup \set{b^i_j}{j<\lambda}$. Finally, $M := \bigcup_{i<\omega} B_i$ is as required, by Tarski-Vaught.
\end{proof}

Thanks to \cref{prop:transitivity}, we also have the following instance of \cite[Theorem~IV.3.2(1)]{Sh-CT}.

\begin{proposition} \label{prop:comp constructible atomic}
  If $B$ is compressibly constructible over $A$, then $B$ is compressible over $A$.
\end{proposition}
\begin{proof}
  Suppose $B = \set{b_i}{i<\alpha}$ is an enumeration witnessing that $B$ is compressibly constructible. Prove by induction on $\beta<\alpha$ that $B_{<\beta} = \set{b_i}{i<\beta}$ is compressible over $A$. For the successor steps use \cref{{prop:transitivity},{lem:set + one more comp over set}}, and for the limit steps use \cref{rem:B compressible iff all finite subsets}.
\end{proof}

So compressible models exist over arbitrary sets. We give some applications.

\subsubsection{Realising models of a stable part}

Suppose $A$ is some small set. Recall that \defn{the induced structure on $A$} is the structure $A_{\ind}$ whose universe is $A$ with the language consisting of a relation $R_{\phi(x)}$ for each $\L$-formula $\phi(x)$, where $R_{\phi(x)}^{A_{\ind}} = \set{a \in A^x}{\U \vDash \phi(a)}$.

\begin{lemma} \label{lem:A stable ->l-isol model}
  Suppose $A$ is such that $\Th(A_{\ind})$ is stable (or just has stable quantifier-free formulas). Then there exists an l-atomic model $M$ over $A$: for every finite tuple $b$ from $M$, $\tp(b/A)$ is l-isolated.
\end{lemma}
\begin{proof}
  By \cref{lem:comp stable type}(i.a), a compressible model over $A$ is l-atomic over $A$, so this follows directly from \cref{prop:comp constructible atomic,{prop:comp constructible}}.
\end{proof}

The next corollary deduces that the reduct map to a stable sort is surjective, and moreover elementary embeddings in the reduct theory can be lifted. Thanks to Anand Pillay and Martin Hils for suggesting that we consider this problem.

\begin{remark}\label{rem:stability of the induced structure on a definable set}
  Suppose $X$ is a $\emptyset$-definable set,
  and let $T_X := \Th(X(\U)_{\ind})$.
  If $T_X$ is stable then $X$ is 
  (uniformly) stably embedded by \cite[Proposition~3.19]{simon-NIP}. It
  follows that stability of $T_X$ is equivalent to assuming that $X$ is stable
  as a partial type in the sense of \cref{def:stable type}.
\end{remark}

\begin{corollary} \label{cor:reduct map surjective}
  Suppose $X$ is a $\emptyset$-definable subset of a sort of $\U$, 
  and let $T_X := \Th(X(\U)_{\ind})$.
  Suppose $T_X$ is stable.
  Let $N \vDash T_X$.
  Then there exists $M \vDash T$ such that $N = X(M)_{\ind}$.

  Moreover, if $N_1 \prec N_2 \vDash T_X$, and $M_1 \vDash T$ is such that $X(M_1)=N_1$, then there is $M_2 \succ M_1$ such that $X(M_2)=N_2$.
\end{corollary}
\begin{proof}
  Since $X(\U)$ is a saturated model of $T_X$, we may assume that $N \prec X(\U)$ and hence $N = A_{\ind}$ where $A\subseteq X(\U)$ is the universe of $N$. By \cref{lem:A stable ->l-isol model}, there exists $M \prec \U$ which is l-atomic over $A$. In particular, $X(M)_{\ind}$ is l-atomic over $A$, and so $X(M) = A$ by \cref{rem:l-isol is realised}.

 Now for the ``moreover'' part. Work in the language $\L(M_1)$ (adding constants for $M_1$). Then $N_1$ can be enriched to a model $N_1' \prec X(\U_{M_1})_{\ind}$ of $T_X' := \Th(X(\U_{M_1})_{\ind})$. By \cref{rem:stability of the induced structure on a definable set}, $X(\U)$ is stably embedded, and since $M_1 \prec \U$ this expansion of $N_1$ is conservative: all new relations are already definable with parameters from $N_1$. 
 Using these definitions, we can enrich $N_2$ to $N_2'$ in such a way that $N_1'\prec N_2'$, and in particular, $N_2' \vDash T_X'$.

  Now, $T_X'$ is still stable, so we may apply the first part (recall that we allow naming any number of constants in $T$) and get some $M_2' \vDash \Th(\U_{M_1})$ such that $X(M_2')=N_2'$. Letting $M_2$ be the reduct to $\L$ (forgetting the constants for $M_1$), we get that $X(M_2)=N_2$ and $M_1 \prec M_2$ as required (this is a bit subtle, since what we really get is that there is an elementary embedding from $M_1$ to $M_2$, but then we can change $M_2$, fixing $N_2$, so that this embedding becomes the identity).
\end{proof}

\begin{remark}
  The assumption that $X$ is a subset of a sort of $\U$ is needed to make sense of the statement (otherwise, if $X$ is e.g., a set of pairs and $N_1$ is not, then it is impossible that $X(M_1)=N_1$). This can be solved by instead asking for an isomorphism between $X(M_i)$ and $N_i$ for $i=1,2$ such that the diagram commutes. This holds, since if $X$ is contained in some product of sorts, then $X$ could be replaced by its image in one sort of $\U^\eq$ under the canonical bijection.
\end{remark}

\begin{example} \label{exa:not always surjective}
 Let $M=(\Aut(\Q,<),\Q,<)$ in the language $\L$ which includes a predicate $Q$ for $\Q$, a predicate $G$ for the group of automorphisms $\Aut(\Q,<)$ with the group operation $\cdot$ and the group action $e:(G\times Q)\to Q$. Let $T=\Th(M)$. Then, the induced structure on $Q$ is just the order, because given two increasing tuples $a:=a_0,\ldots, a_{n-1}$, $b:=b_0,\ldots, b_{n-1}$, there is an automorphism $\sigma$ of $M$ mapping $a$ to $b$ (there is some automorphism $g \in \Aut(\Q,<)$ mapping $a$ to $b$; let $\sigma$ be $g$ on $Q$ and conjugation by $g$ on $G$). 
 Hence, the theory $\Th(Q_{\ind})$ is a conservative extension of DLO $=\Th(\Q,<)$: all the relations are definable from the order.

 However, $T$ says that any two intervals have the same cardinality (because there is a bijection between them), and there are models of DLO in which this is not the case, and hence \cref{cor:reduct map surjective} does not hold in this case. This is not a surprise, since $Q$ is not stably embedded (as the graph of any non-trivial automorphism is not definable using just the order, since $\dcl$ is trivial there), and moreover the theory is not NIP (for any finite tuple, there is an automorphism fixing any sub-tuple and moving the rest).
\end{example}

\cref{exa:not always surjective} raises the following question.

\begin{question}
  Suppose that $T$ is NIP and that $X \subseteq \U$ is $\emptyset$-definable and stably embedded. Is it true that every model $M$ of $\Th(X(\U)_{\ind})$ of the form $X(M)$ for some $M\vDash T$?
\end{question}

This question is related to Gaifman's categoricity conjecture, see Hodges \cite{Hodges-GaifmanConj,Gaifman}.

\subsubsection{Extensions preserving a stable part}

\begin{corollary} \label{cor:inf orders => arbitrary large comp models}
  Suppose $A$ is a set such that if $M$ is a model containing $A$, then $M$ contains an infinite chain in some $M$-definable preorder.

  Then there are arbitrarily large models which are compressible over $A$.
\end{corollary}
\begin{proof}
  There exists a compressible model over $A$ by \cref{prop:comp constructible atomic,{prop:comp constructible}}. It is enough to show that any such model $M$ can be strictly extended to a compressible model over $A$. By assumption $M$ contains an infinite chain in some definable preorder. By \cref{rem:enough infinite chain}, there is some $c$ such that $\tp(c/M)$ is compressible but not l-isolated, so $c \notin M$. Apply (the ``moreover'' part of) \cref{prop:comp constructible} to get some compressibly constructible model $N$ over $M$ containing $Mc$. Then $N$ is compressible over $M$ by \cref{prop:comp constructible atomic}. By transitivity (\cref{prop:transitivity}), $N$ is compressible over $A$.
\end{proof}

\begin{example}
  The following example shows that \cref{cor:inf orders => arbitrary large comp models} requires the assumption that every model containing $A$ contains an infinite chain; mere instability does not suffice. Let $\L = \{<\}$ where $<$ is a binary relation symbol. Let $T$ say that $<$ is a partial order, that comparability is an equivalence relation having exactly one class of size $n$ for each $0<n<\omega$, and that each class has maximal and minimal elements and is discretely ordered (every non-maximal element has a successor and every non-minimal element has a predecessor). Let $M$ be the $\L$-structure with universe $\set{(i,j)}{i\in \N,j\leq i}$ and with $<^M = \set{(i,j_1),(i,j_2) \in M^2}{i\in \N,j_1<j_2}$. Then $M \vDash T$, and in fact every model of $T$ is a disjoint union of (a copy of) $M$ and infinite chains, each discretely ordered with a minimal and maximal element. Hence any two saturated models of $T$ are isomorphic and thus $T$ is complete. Moreover, $T$ has quantifier elimination after adding the predecessor and successor functions by a back-and-forth argument. It follows that $T$ is NIP (by e.g.\ showing that $x<y$ is NIP as there is a polynomial bound on every $(x<y)$-type over a finite subset of $M$).

  Suppose $c\in \U$. Then $\tp(c/M)$ is generically stable (see \cref{def:GS types over models}). Indeed, suppose $c\notin M$ and let $p$ be a global coheir extending $\tp(c/M)$. Then any Morley sequence $I:=\sequence{a_i}{i<\omega}$ of $p$ over $M$ must be such that $a_i,a_j$ are incomparable for $i\neq j$. 
  Thus $I$ is an indiscernible set by quantifier elimination, and hence $p$ is generically stable by NIP and the ``moreover'' part of \cref{fac:generically stable -> dfs}.

  Thus, if $\tp(c/M)$ is compressible then it is l-isolated by \cref{prop:generically stable compressible types are l-isolated} and hence realised by \cref{rem:l-isol is realised}. Hence $M$ is the only model compressible over $M$ and the conclusion of \cref{cor:inf orders => arbitrary large comp models} does not hold.

\end{example}

\begin{corollary} \label{cor:extend Preserve Stable}
  Let $M \vDash  T$, and suppose that $<$ is an $M$-definable preorder on an $M$-definable set $D$ and suppose that there is an infinite $<$-chain in $D(M)$.
  Then:

  \begin{enumerate}[(i)]
    \item There exists $N \succ  M$ with $D(N) \supsetneq D(M)$ such that if $a\in N$ and $\tp(a/M)$ is generically stable, then $a \in M$. (See \cref{def:GS types over models}.)
    \item For any cardinal $\lambda$ there is $N \succ  \M$ with $|D(N)| \geq  \lambda$ such that if $a\in N$ and $\tp(a/M)$ is generically stable, then $a\in M$.

    In particular, there are arbitrarily large elementary extensions $N$ of $M$ such that if $S$ is $M$-definable and stable (see \cref{def:stable type}), then $S(M) = S(N)$.
  \end{enumerate}


\end{corollary}
\begin{proof}
  (ii) follows from (i) by taking the union of a suitably long elementary chain. We prove (i).

  By \cref{rem:enough infinite chain} there is some $c \in D(\U)$ such that $\tp(c/M)$ is compressible but not l-isolated. In particular $c \notin M$. Apply (the ``moreover'' part of) \cref{prop:comp constructible} to get some compressibly constructible model $N \succ M$ over $M$ containing $Mc$. By \cref{prop:generically stable compressible types are l-isolated}, if $a \in N$ and $\tp(a/M)$ is generically stable, then it is l-isolated and hence by \cref{rem:l-isol is realised}, it is realised.

  The ``in particular'' part follows since for any $a \in S(\U)$, $\tp(a/M)$ is stable and hence generically stable by \cref{fac:stable => GS over model}.
\end{proof}

\begin{remark}\label{rem:extend preserve stable omega-sat}
  Note that the condition in \cref{cor:extend Preserve Stable} holds whenever $T$ is unstable and $M$ is $\omega$-saturated, since in this case $T$ has the SOP (\cite[Theorem~2.67]{simon-NIP}) and hence any $\omega$-saturated model contains an infinite chain in some definable preorder.
\end{remark}

\subsubsection{Compressible types in ACVF}

\newcommand{\ACVF}{\operatorname{ACVF}}
\newcommand{\res}{\operatorname{res}}
\newcommand{\alg}{{\operatorname{alg}}}
\begin{example}\label{exa:ACVF}
  If $\U \vDash \ACVF$ and $A \subseteq \U^\eq$,
  then any model $M$ containing $A$ whose residue field is algebraic over $A$ is compressible over $A$; i.e.,
  if $M \prec \U$, $A \subseteq M^\eq$, and $k(M) = \acleq(A) \cap k$,
  where $k = k(\U)$ denotes the residue field,
  then $M$ is compressible over $A$.

  To see this, consider first a 1-type $\tp(a/B)$ over an $\acleq$-closed set $B=\acleq(B) \subseteq \U^\eq$, where $a \in \U$. We show that $\tp(a/B)$ is compressible iff $\dcleq(Ba) \cap k = B \cap k$.
  By unique Swiss cheese decompositions \cite[Theorem~3.26]{Holly}, $\tp(a/B)$ is determined by knowing which of the balls defined over $B$ contain $a$, i.e., $\tp(a/B)$ is implied by $\tp_\in(a/B)$ where $x \in y$ is the element relation between the valued field and the sort of balls (closed and open).
  One sees directly that this $\in$-type is compressible unless $\tp(a/B)$ is the generic type of a closed ball $\alpha$ over $B$ which contains infinitely many open balls over $B$ of radius $\gamma := \operatorname{radius}(\alpha)$.
  In that case, let $\beta_1 \neq \beta_2$ be distinct open subballs of $\alpha$ over $B$ of radius $\gamma$ (the existence of these two balls is our only use of our assumption that there are infinitely many such), and consider the map $\beta \mapsto \res(\frac{\beta-\beta_1}{\beta_2-\beta_1})$ for $\beta$ an open subball of $\alpha$ of radius $\gamma$. This is a well-defined injective map to $k$ defined over $B$, so genericity of $\tp(a/B)$ (and the fact that $B=\acleq(B)$) implies that the image under this map of the open ball around $a$ of radius $\gamma$ is not in $B \cap k$, so we have $\dcleq(Ba) \cap k \supsetneq B \cap k$.
  Conversely, if $\tp(a/B)$ is compressible, then since the residue field is a pure stably embedded algebraically closed field, we must have $\dcleq(Ba) \cap k = B \cap k$.

  The compressibility of $M$ over $A$ if $k(M) = \acleq(A) \cap k$ now follows by first taking a compressible construction sequence which alternates taking $\acleq$-closure and adding a single new element of $M$, and then applying \cref{prop:comp constructible atomic}.

  If $A \cap k$ is infinite then conversely $k(M) = \acleq(A) \cap k$ is necessary for compressibility of $M$ over $A$, but this fails in general;
  for example, if $A$ is finite, then any model containing $A$ is of course compressible over $A$.
\end{example}

\begin{remark}\label{rem:BM}
  The argument in \cref{exa:ACVF} is based on ideas from \cite{BM}, and combining it with \cite[Remark~3.12]{BM} actually yields an alternative proof of \cite[Theorem~5.6]{BM}. Indeed, suppose $K$ is a valued subfield of $\U$ with $\res(K)$ finite. Taking an elementary extension we may assume that $(\U,K)$ is sufficiently saturated. Let $A \subseteq K$. Then $k(K^\alg) = \acleq(A) \cap k$ is the algebraic closure of the prime field, so by \cref{exa:ACVF}, $K^\alg$, and in particular $K$, is compressible over $A$. By \cite[Remark~3.12]{BM}, it follows that $K$ is distal in $\U$ in the sense defined in that paper. The same argument applies to $K_r$ defined in \cite[Theorem~5.6]{BM}. However, this proof does not yield bounds on the exponents of the resulting distal cell decompositions, whereas such bounds are obtained in \cite[Remark~6.2]{BM} as a consequence of the more elementary methods of that paper.
\end{remark}

\bibliographystyle{alpha}
\bibliography{denseCompression}

\begin{thebibliography}{HWLW17}

\bibitem[ACP14]{Gen}
Hans Adler, Enrique Casanovas, and Anand Pillay.
\newblock Generic stability and stability.
\newblock {\em J. Symb. Log.}, 79(1):179--185, 2014.

\bibitem[BM21]{BM}
Martin Bays and Jean-François Martin.
\newblock Incidence bounds in positive characteristic via valuations and
  distality.
\newblock 2021.
\newblock arXiv:2104.04339.

\bibitem[CCT16]{CCT}
Xi~Chen, Yu~Cheng, and Bo~Tang.
\newblock On the recursive teaching dimension of {VC} classes.
\newblock In Daniel~D. Lee, Masashi Sugiyama, Ulrike von Luxburg, Isabelle
  Guyon, and Roman Garnett, editors, {\em Advances in Neural Information
  Processing Systems 29: Annual Conference on Neural Information Processing
  Systems 2016, December 5-10, 2016, Barcelona, Spain}, pages 2164--2171, 2016.

\bibitem[CG20]{MR4033642}
Gabriel Conant and Kyle Gannon.
\newblock Remarks on generic stability in independent theories.
\newblock {\em Ann. Pure Appl. Logic}, 171(2):102736, 20, 2020.

\bibitem[CS15]{CS-extDef2}
Artem Chernikov and Pierre Simon.
\newblock Externally definable sets and dependent pairs {II}.
\newblock {\em Trans. Amer. Math. Soc.}, 367(7):5217--5235, 2015.

\bibitem[DKL84]{MR876078}
Richard~M. Dudley, Hiroshi Kunita, and François Ledrappier.
\newblock {\em \'{E}cole d'\'{e}t\'{e} de probabilit\'{e}s de {S}aint-{F}lour.
  {XII}---1982}, volume 1097 of {\em Lecture Notes in Mathematics}.
\newblock Springer-Verlag, Berlin, 1984.
\newblock Papers from the summer school held in Saint-Flour, August
  22--September 8, 1982, Edited by P. L Hennequin.

\bibitem[EK20]{UDTFS}
Shlomo Eshel and Itay Kaplan.
\newblock On uniform definability of types over finite sets for {NIP} formulas.
\newblock {\em Journal of Mathematical Logic}, 0(0):2150015, 2020.

\bibitem[EK21]{MR4216280}
Pedro~Andr\'{e}s Estevan and Itay Kaplan.
\newblock Non-forking and preservation of {NIP} and dp-rank.
\newblock {\em Ann. Pure Appl. Logic}, 172(6):102946, 2021.

\bibitem[Gai74]{Gaifman}
Haim Gaifman.
\newblock Operations on relational structures, functors and classes. {I}.
\newblock In {\em Proceedings of the {T}arski {S}ymposium ({P}roc. {S}ympos.
  {P}ure {M}ath., {V}ol. {XXV}, {U}niv. {C}alifornia, {B}erkeley, {C}alif.,
  1971)}, pages 21--39, 1974.

\bibitem[Gui12]{MR2963018}
Vincent Guingona.
\newblock On uniform definability of types over finite sets.
\newblock {\em J. Symbolic Logic}, 77(2):499--514, 2012.

\bibitem[Hod11]{Hodges-GaifmanConj}
Wilfrid Hodges.
\newblock Proof of {G}aifman's conjecture for relatively categorical abelian
  groups.
\newblock {\em An. \c{S}tiin\c{t}. Univ. ``Ovidius'' Constan\c{t}a Ser. Mat.},
  19(2):101--130, 2011.

\bibitem[Hol95]{Holly}
Jan~E. Holly.
\newblock Canonical forms for definable subsets of algebraically closed and
  real closed valued fields.
\newblock {\em J. Symbolic Logic}, 60(3):843--860, 1995.

\bibitem[HP11]{MR2800483}
Ehud Hrushovski and Anand Pillay.
\newblock On {NIP} and invariant measures.
\newblock {\em J. Eur. Math. Soc. (JEMS)}, 13(4):1005--1061, 2011.

\bibitem[HWLW17]{k-isolationQuadratic}
Lunjia Hu, Ruihan Wu, Tianhong Li, and Liwei Wang.
\newblock Quadratic upper bound for recursive teaching dimension of finite {VC}
  classes.
\newblock In Satyen Kale and Ohad Shamir, editors, {\em Proceedings of the 30th
  Conference on Learning Theory, {COLT} 2017, Amsterdam, The Netherlands, 7-10
  July 2017}, volume~65 of {\em Proceedings of Machine Learning Research},
  pages 1147--1156. {PMLR}, 2017.

\bibitem[Mat04]{MR2060639}
Ji\v{r}\'{\i} Matou\v{s}ek.
\newblock Bounded {VC}-dimension implies a fractional {H}elly theorem.
\newblock {\em Discrete Comput. Geom.}, 31(2):251--255, 2004.

\bibitem[Pil96]{PillayGeometric}
Anand Pillay.
\newblock {\em Geometric stability theory}, volume~32 of {\em Oxford Logic
  Guides}.
\newblock The Clarendon Press, Oxford University Press, New York, 1996.
\newblock Oxford Science Publications.

\bibitem[PT11]{AnandPredrag}
Anand Pillay and Predrag Tanovi\'{c}.
\newblock Generic stability, regularity, and quasiminimality.
\newblock In {\em Models, logics, and higher-dimensional categories}, volume~53
  of {\em CRM Proc. Lecture Notes}, pages 189--211. Amer. Math. Soc.,
  Providence, RI, 2011.

\bibitem[She90]{Sh-CT}
S.~Shelah.
\newblock {\em Classification theory and the number of nonisomorphic models},
  volume~92 of {\em Studies in Logic and the Foundations of Mathematics}.
\newblock North-Holland Publishing Co., Amsterdam, 1990.

\bibitem[She04]{MR2062198}
Saharon Shelah.
\newblock Classification theory for elementary classes with the dependence
  property---a modest beginning.
\newblock volume~59, pages 265--316. 2004.
\newblock Special issue on set theory and algebraic model theory.

\bibitem[She09]{MR2570659}
Saharon Shelah.
\newblock Dependent first order theories, continued.
\newblock {\em Israel J. Math.}, 173:1--60, 2009.

\bibitem[Sim15a]{simon-NIP}
Pierre Simon.
\newblock {\em A guide to {NIP} theories}, volume~44 of {\em Lecture Notes in
  Logic}.
\newblock Association for Symbolic Logic, Chicago, IL; Cambridge Scientific
  Publishers, Cambridge, 2015.

\bibitem[Sim15b]{simon-invariant}
Pierre Simon.
\newblock Invariant types in {NIP} theories.
\newblock {\em J. Math. Log.}, 15(2):1550006, 26, 2015.

\bibitem[Sim20]{simon-decomposition}
Pierre Simon.
\newblock Type decomposition in {NIP} theories.
\newblock {\em J. Eur. Math. Soc. (JEMS)}, 22(2):455--476, 2020.

\bibitem[TZ12]{TentZiegler}
Katrin Tent and Martin Ziegler.
\newblock {\em A course in model theory}, volume~40 of {\em Lecture Notes in
  Logic}.
\newblock Association for Symbolic Logic, La Jolla, CA; Cambridge University
  Press, Cambridge, 2012.

\bibitem[Usv09]{MR2499428}
Alexander Usvyatsov.
\newblock On generically stable types in dependent theories.
\newblock {\em J. Symbolic Logic}, 74(1):216--250, 2009.

\bibitem[V{\"{a}}{\"{a}}03]{MR2007169}
Jouko V{\"{a}}{\"{a}}n{\"{a}}nen.
\newblock Pseudo-finite model theory.
\newblock volume~24, pages 169--183. 2003.
\newblock 8th Workshop on Logic, Language, Informations and
  Computation---WoLLIC'2001 (Bras\'{\i}lia).

\end{thebibliography}

\end{document}